\documentclass[11pt]{amsart}
\usepackage{amsmath,amsthm,amsfonts,amssymb, mathrsfs, wasysym,hyperref} 
\usepackage{graphicx}
\usepackage[cmtip,arrow]{xy}
\usepackage{pb-diagram, pb-xy}

\newtheorem{thmi}{Theorem}
\newtheorem{cori}[thmi]{Corollary}
\newtheorem{thm}{Theorem}[section]
\newtheorem{theo}[thm]{Theorem}

\newtheorem{cor}[thm]{Corollary}

\newtheorem{prop}[thm]{Proposition} 
 
\newtheorem{lem}[thm]{Lemma} 
\newtheorem{lemma}[thm]{Lemma}

 \theoremstyle{definition}
 
\newtheorem{defi}[thm]{Definition}

\newtheorem{rem}[thm]{Remark}
\newtheorem{remark}[thm]{Remark}

\newtheorem{notation}[thm]{Notation}
\newtheorem{assumption}[thm]{Standing assumptions}

\makeatletter
\@tfor\next:=abcdefghijklmnopqrstuvwxyzABCDEFGHIJKLMNOPQRSTUVWXYZ\do{%
  \def\command@factory#1{%
    \expandafter\def\csname cal#1\endcsname{\mathcal{#1}}
    \expandafter\def\csname frak#1\endcsname{\mathfrak{#1}}
    \expandafter\def\csname scr#1\endcsname{\mathscr{#1}}
    \expandafter\def\csname bb#1\endcsname{\mathbb{#1}}
    \expandafter\def\csname rm#1\endcsname{\mathrm{#1}}
      \expandafter\def\csname bf#1\endcsname{\mathbf{#1}}
  }
 \expandafter\command@factory\next
}
\makeatother

\newcommand{\Cay}{\operatorname{Cay}}

\newcommand{\Act}{\operatorname{Act}}
\newcommand{\G}{\Gamma}
\newcommand{\ang}{\sphericalangle}

\begin{document}

\title[Dehn filling Dehn twists]{Dehn filling Dehn twists}
\author[F. Dahmani]{Fran\c{c}ois Dahmani}
\address{Institut Fourier, Univ. Grenoble Alpes, CNRS, Grenoble, France}
\email{francois.dahmani@univ-grenoble-alpes.fr}
\thanks{}
\author[M. Hagen]{Mark Hagen}
\address{School of Mathematics, Univ. Bristol, Bristol, United Kingdom}
\email{markfhagen@posteo.net}
\thanks{Hagen was supported by EPSRC grant EPSRC EP/R042187/1}
\author[A. Sisto]{Alessandro Sisto}
\address{Department of Mathematics, ETH Zurich, Zurich, Switzerland}
\email{sisto@math.ethz.ch}
\thanks{}

\maketitle

\begin{abstract}
Let $\Sigma_{g,p}$ be the genus--$g$ oriented surface with $p$ punctures, with either $g>0$ or $p>3$.  We show that $MCG(\Sigma_{g,p})/DT$ is 
acylindrically hyperbolic where $DT$ is the normal subgroup of the mapping class group 
$MCG(\Sigma_{g,p})$ generated by $K^{th}$ powers of Dehn twists about curves in $\Sigma_{g,p}$ for suitable $K$.

Moreover, we show that in low complexity $MCG(\Sigma_{g,p})/DT$ is in fact hyperbolic. In particular, for $3g-3+p\leq 2$, we show that the mapping class group $MCG(\Sigma_{g,p})$ is fully residually 
non-elementary hyperbolic and admits an affine isometric action with unbounded 
orbits on some $L^q$ space.  Moreover, if every hyperbolic group is residually 
finite, then every convex-cocompact subgroup of $MCG(\Sigma_{g,p})$ is 
separable. 

The aforementioned results follow from general theorems about composite rotating families, 
in the sense of~\cite{D_PlateSpinning}, that come from a collection of subgroups of 
vertex stabilisers for the action of a group $G$ on a hyperbolic graph $X$. We give conditions ensuring that the graph $X/N$ is 
again hyperbolic and various properties of the action of $G$ on $X$ persist for 
the action of $G/N$ on $X/N$.
\end{abstract}

\tableofcontents

\section*{Introduction}
Thurston's Dehn filling theorem has an algebraic counterpart in the
context of relatively hyperbolic groups
\cite{Os-perfill,GrMa-perfill}, which has numerous important
applications such as in the proof of the Virtual Haken conjecture
\cite{vhak} and in the solution of the isomorphism problem for certain
relatively hyperbolic groups \cite{DG:recognize_DF, DT-DecidingIsom}. A Dehn filling of a relatively hyperbolic group $G$ is the quotient of $G$ by the normal closure of normal subgroups $N_i\triangleleft H_i $ of its peripheral subgroups $H_i$, and the Dehn filling theorem for relatively hyperbolic groups says that such quotients are still relatively hyperbolic provided that the $N_i$ are sufficiently ``sparse''.  Mapping class groups are not non-trivially relatively hyperbolic except in very low complexity \cite{AAS:MCG_not_rel_hyp, BDM:thick}, but it is still natural to think of their subgroups generated by Dehn twists around curves in a pants decomposition as peripheral subgroups. Hence, we think of the following theorem (Theorem~\ref{thm:mcg_acyl} below), as a Dehn filling theorem for mapping class groups: 

\begin{thmi}\label{thmi:acyl}
 Suppose $g>0$ or $p>3$, and consider  $\Sigma_{g,p}$, 
 an oriented surface of genus $g$ with $p$ punctures. There exists a positive integer $K_0$ so
 that for all non-zero multiples $K$ of $K_0$,  if  $DT_K$ denotes the
 normal subgroup of the mapping class group $MCG(\Sigma_{g,p})$ generated by all $K^{th}$ 
powers of Dehn twists, then  the group $MCG(\Sigma_{g,p})/DT_K$ is
acylindrically hyperbolic.  
\end{thmi}

\subsection*{Residual properties of mapping class groups in low complexity}
Recall that the complexity of $\Sigma_{g,p}$ is defined as $\xi(g,p)=3g+p-3$. 
In low 
complexity, we can press a bit further and study residual properties of mapping 
class groups.  First, we produce many hyperbolic quotients. Recall that, given a class of groups $\mathcal P$, a group $G$ is fully residually $\mathcal P$ if for every finite set $F\subseteq G-\{1\}$ there exists a group $H$ in $\mathcal P$ and a surjective homomorphism $\phi:G\to H$ with $\phi(F)\subseteq H-\{1\}$.

\begin{thmi}\label{thmi:resid}
Suppose that $(g,p)\in\{(0,5),(1,2)\}$. Then $MCG(\Sigma_{g,p})$ is fully residually non-elementary 
hyperbolic.
\end{thmi}

It was not previously known whether either of $MCG(\Sigma_{0,5})$ or 
$MCG(\Sigma_{1,2})$ admits an infinite 
hyperbolic quotient.

We deduce 
Theorem~\ref{thmi:resid} from a different statement, 
Theorem~\ref{thm:mcg_quotients}, which is interesting even in complexity $1$.  
Specifically, for $(g,p)\in\{(0,4),(1,0),(1,1)\}$, this says that 
$MCG(\Sigma_{g,p})/DT_K$ is hyperbolic for all suitably large multiples $K$ of the constant $K_0$ from 
Theorem~\ref{thm:mcg_quotients}, and when 
$(g,p)\in\{(0,5),(1,2)\}$, the quotient $MCG(\Sigma_{g,p})/DT_K$ is hyperbolic relative 
to subgroups isomorphic to $MCG(\Sigma_{g',p'})/DT_K$ with $(g',p')\in\{(0,4),(1,1)\}$ 
and is therefore again hyperbolic.

From Theorem~\ref{thmi:resid} and results of Yu and Nica (see
also Alvarez and Lafforgue), we obtain:

\begin{cori}\label{cori:l_q}
 Suppose that $(g,p)\in\{(0,5),(1,2)\}$. Then $MCG(\Sigma_{g,p})$ admits an affine isometric action 
with unbounded orbits on some $L^q$ space.
\end{cori}

Specifically, Theorem~\ref{thmi:resid} yields a non-elementary hyperbolic 
quotient $Q$, which in turn admits a proper affine isometric action on 
$L^q(\partial Q\times \partial Q)$ and $\ell^q(Q\times Q)$,
by~\cite{Nica,Yu, AlLa}, 
for sufficiently large $q$. 

This is related to the 
question of which mapping class groups have property (T) because, if $Q$ can be 
chosen as above so that one could take $q=2$, we would get an affine isometric 
action of $MCG(\Sigma_{g,p})$, with unbounded orbits, on a Hilbert space.

Recall from~\cite{FarbMosher} that a subgroup $H\leq MCG(\Sigma_{g,p})$ is 
\emph{convex-cocompact} if some (hence any) $H$--orbit in the Teichm\"uller 
space of $\Sigma_{g,p}$ is quasiconvex.  
There are several equivalent characterisations of convex-cocompactness, see~\cite{KentLeininger:convex_cocompact, Hamenstadt, 
DurhamTaylor:stability},  and one reason this notion is interesting is its connection with hyperbolicity of fundamental 
groups of surface bundles over surfaces \cite{FarbMosher,Hamenstadt}.

In~\cite{Reid:separable}, Reid posed the question of whether convex-cocompact subgroups of 
$MCG(\Sigma_{g,p})$ are separable. Recall that a subgroup $H<G$ is separable if for every $x\in G-H$ there exists a finite group $F$ a surjective homomorphism $\phi:G\to F$ with $\phi(x)\notin \phi(H)$.

Note that in general $MCG(\Sigma_{g,p})$ contains non-separable subgroups, and in fact this is already the case for 
$MCG(\Sigma_{0,5})$~\cite{LeiningerMcreynolds}.   (Nonetheless, various 
geometrically natural subgroups, e.g. curve-stabilisers, are known to be 
separable in $MCG(\Sigma_{g,p})$~\cite{LeiningerMcreynolds}.)  

The techniques of the present paper engage with this question in a somewhat 
mysterious way:

\begin{thmi}\label{thmi:convex_cocompact}
 Assume that all hyperbolic groups are residually finite.  Suppose 
$(g,p)\in\{(0,5),(1,2)\}$. Then any convex-cocompact subgroup 
$Q<MCG(\Sigma_{g,p})$ is separable.
\end{thmi}

The proof of Theorem~\ref{thmi:convex_cocompact} relies on the hyperbolic 
quotients of $MCG(\Sigma_{g,p})$ arising in the proof of 
Theorem~\ref{thmi:resid}.  The extra work, done in 
Proposition~\ref{prop:convex_cocompact}, is to show that in all but finitely 
many such quotients, the subgroup $Q$ survives as a quasiconvex subgroup, and 
arbitrary $g\in\Sigma_{g,p}\setminus Q$ can be separated from $Q$ in 
such quotients.  At this point, the assumption about residual finiteness is 
invoked: using a result of~\cite{AGM} (namely, if all hyperbolic groups are residually finite then all hyperbolic groups are 
QCERF), we then separate the image of $g$ from this 
quasiconvex subgroup in a finite quotient.

\subsection*{More general context and proof strategy}
In order to prove our results we actually work in a more general context (that we plan on using in future work, see the next 
subsection).  Roughly speaking, we consider a group $G$ acting on a hyperbolic graph $X$ so that the vertex set of $X$ fits 
the framework of composite projection systems introduced in \cite{D_PlateSpinning}, and consider quotients $G/N$ of $G$ by 
normal subgroups $N$ generated by subgroups of vertex stabilisers consisting of ``big rotations''. The main technical 
innovation we introduce in this paper is our method for proving that the graph $X/N$ is hyperbolic. The strategy is as 
follows (see Proposition \ref{prop:lift}). Consider a geodesic triangle in $X/N$. We can lift the 3 sides of the triangle to 
a concatenation of 3 geodesics, which need not close up. If it doesn't close up, there is a non-trivial element $n\in N$ that 
maps the initial vertex $x$ of the concatenation to the terminal vertex $nx$. We then want to change the lifts so that the 
new concatenation is either a triangle, or at least $n$ is ``simpler''. What allows us to do this is Corollary 
\ref{cor:short_less_complex}, which we think of as an analogue of the Greendlinger lemma from \cite{DGO}. What the corollary 
says is that we have a large rotation $\gamma_v$ around some vertex $v$ so that $\gamma_vn$ is ``simpler'', and $v$ 
needs to lie within distance $1$ of one of the lifts. Applying $\gamma_v$ to part of the concatenation yields new lifts with 
the required property. The measure of complexity of elements of $N$ is actually rather complicated, but the same proof 
strategy works in other contexts as long as there is a version of the Greendlinger lemma so that some notion of complexity 
gets reduced when applying it. In particular, it can replace the use of the Cartan-Hadamard Theorem (for hyperbolic space) in 
the context of very rotating families from \cite{DGO} (thereby making for a simpler proof). On the other hand, 
Cartan-Hadamard cannot be applied to $X/N$ in our context, and in fact we could not see any way to apply it to any space 
quasi-isometric to it.

\subsection*{Future work in the hierarchically hyperbolic setting}
A natural strategy for extending the above applications beyond complexity $2$ 
involves combining the techniques in the present paper with the theory of 
\emph{hierarchically hyperbolic groups}~\cite{HHS_I,HHS_II} (of which mapping 
class 
groups are one of the ``type species''). 

We believe that the quotients $MCG(\Sigma_{g,p})/DT_K$ are in fact 
hierarchically hyperbolic.  One could then apply 
again our results on composite rotating graphs, and take further quotients. The complexity (in the hierarchically 
hyperbolic sense) decreases at each step, and hierarchically hyperbolic groups of minimal complexity are known to be 
hyperbolic \cite{HHS:quasiflats}.
In particular, under the assumption that every hyperbolic group is residually 
finite, one should be able to prove that for arbitrary $g,p$, the group $MCG(\Sigma_{g,p})$ admits a non-elementary 
hyperbolic quotient, and that every convex-cocompact subgroup is separable.

\subsection*{Outline of the paper}
In Section~\ref{sec:CPS_CRF}, we recall the notions of \emph{composite projection systems} and \emph{composite rotating 
families} from~\cite{D_PlateSpinning} and establish some useful facts, relying on the \emph{transfer lemma} 
from~\cite{D_PlateSpinning}.  The notion of a composite projection system relies on the projection axioms from~\cite{BBF}.  

In Section~\ref{sec:composite_projection_graph}, we introduce hyperbolic graphs into the picture, and define the notion of a 
composite projection graph and a composite rotating family on it.  Here, we state our main technical result, 
Theorem~\ref{thm:quotient_proj_graph}, and give a proof which relies on statements proved in subsequent sections.

The reader mainly interested in the proof of Theorem~\ref{thm:quotient_proj_graph} is advised to focus on the main result of 
Section~\ref{sec:shortening}, which is Corollary~\ref{cor:short_less_complex}; this is the statement, mentioned above, that 
allows one to lower the ``complexity'' of elements $n\in N$ by applying large rotations.  This is used 
in Section~\ref{sec:lift_and_project} to construct the lifts mentioned above.  From the lifting procedure, one obtains the 
facts used in the proof of Theorem~\ref{thm:quotient_proj_graph}.  In Section~\ref{sec:lift_and_project}, we also prove 
Proposition~\ref{prop:stab}, which describes how stabilisers of vertices in $X$ intersect $N$.  This is not used in the 
proof of Theorem~\ref{thm:quotient_proj_graph}, but does play a role in Section~\ref{sec:MCG}.

Finally, in Section~\ref{sec:MCG}, we consider the case of $MCG(\Sigma_{g,p})$ acting on the curve graph $\mathcal 
C(\Sigma_{g,p})$, which is a composite projection graph by Proposition~\ref{prop:mcg_cpg} (see also~\cite{D_PlateSpinning}). 
 Large powers of Dehn twists generate a collection of rotation subgroups forming a composite rotating family, and we can 
invoke Theorem~\ref{thm:quotient_proj_graph} to obtain Theorem~\ref{thmi:acyl}.  The rest of Section~\ref{sec:MCG} is 
devoted to the proofs of Theorem~\ref{thmi:resid} and 
Theorem~\ref{thmi:convex_cocompact}.

\section{Composite projection systems and rotating families}\label{sec:CPS_CRF}
We now recall the notion of a composite projection system from \cite{D_PlateSpinning}, and establish some basic facts.  The 
reader familiar with mapping class groups might want to keep in 
mind that in that context $\bbY_*$ is the collection of (isotopy classes of simple closed) curves, that two curves are 
active if they intersect, and that $d^\pi_y$ is defined using subsurface projection.

\subsection{Composite projection systems}\label{subsec:CPS}
Given $y$ in a partitioned set $\bbY_* = \sqcup_{i=1}^m \bbY_i$, denote by 
$i(y)$ the index such that $y\in  \bbY_{i(y)}$. 

\begin{defi}\label{def;CPS}\cite[Definition 1.2]{D_PlateSpinning}
	Let $\bbY_*$ be the disjoint union of finitely many countable
	sets         $\bbY_1, \dots, \bbY_m$. A \emph{composite projection system} on 
	$ (\bbY_i)_{i=1 .. m}$ (or on $\bbY_*$),  for the constant $\theta$,
	consists of 
	\begin{itemize}   
		\item a family of subsets  $\Act(y)\subset   \bbY_*$ for $ y\in
		\bbY_*$  (the \emph{active set} for $y$) such that $\bbY_{i(Y)} \subset\Act(y)$,     and such that  $x\in 
\Act(y)$ if and
		only if $y\in \Act(x)$  (\emph{symmetry in action}), 
		\item and  a family of functions
		$d^\pi_y : ( \Act(y)\setminus \{y\} )^2   \to \bbR_+$, satisfying: 
		\begin{itemize} 
		\item \textbf{Symmetry}: $d^\pi_y(x,z)=d^\pi_y(z,x)$ for $x,z\in\Act(y)\setminus\{y\}$;
		\item \textbf{Triangle inequality}: $d^\pi_y(w,x)\leq d^\pi_y(w,z)+d^\pi_y(z,x)$ for all 
$w,x,z\in\Act(y)\setminus\{y\}$;
		\item \textbf{Behrstock inequality:} $\min\{d^\pi_y(x,z),d^\pi_z(x,y)\}\leq\theta$ whenever both quantities 
are defined; 
			\item \textbf{Properness:} $|\{y \in \bbY_i, 
d^\pi_y(x,z)\geq \theta \}|<\infty$ for all $x,z \in \bbY_i$;
			\item  \textbf{Separation:} $d^\pi_y(z,z)<\theta$ for $z\in\Act(y)\setminus\{y\}$;
			\item  \textbf{Closeness in inaction:} if   $x\notin \Act(z)$ then, for
			all $y \in \Act(x) \cap \Act(z)$, we have $ d_y^\pi(x, z)\leq \theta$; 
			\item \textbf{Finite filling:}  for all $\calZ \subset \bbY_*$, there
			is a finite collection $x_j\in\calZ$ such that $\cup_j
			\Act(x_j)$ covers  $\cup_{ x\in \calZ} \Act(x)$.
		\end{itemize}
	\end{itemize}
\end{defi}

By \cite[Theorem 3.3]{BBF}, for each 
$i\leq m$, and $y\in \bbY_i$, and  for a suitable choice of $\theta$, there exists  a
modified function $d_y : \bbY_{i} \times \bbY_i       \to  \bbR_+$, 
satisfying the \emph{monotonicity} property of \cite[Theorem 3.3]{BBF} (see also \cite[Axiom (SP3)]{BBFS} for a strengthened property).  

This function is unfortunately not defined on
$\Act(y)\setminus \bbY_i$. However, $d^\pi_y$ is defined on all of $\Act(y)$.  Therefore, we define 
$d^\ang_y:(\Act(y)\setminus\{y\})^2\to\bbR_+$ as follows.  Let $x,z\in\Act(y)\setminus\{y\}$.  If $x,z\in\bbY_i$, we let 
$d^\ang_y(x,z)=d_y(x,z)$, and otherwise, we let $d^\ang_y(x,z)=d^\pi_y(x,z)$. 

Let $\bbY^j_M (x, z)=\{y\in \bbY_j \cap \Act(x)
\cap \Act(z),  d^\ang_y(x, z) \geq M\}$ (the set of $M$-large projections between $x$ and $z$ in the $j$-coordinate).  The 
elements $x,y,z$ need not be in the same coordinate.

We now introduce the first of various constraints on the constants that will appear.  Fix a composite projection system with 
constant $\theta$.  Let $\Theta=\Theta(\theta)$ be the constant provided by applying Theorem~3.3 of~\cite{BBF} to each 
$\bbY_i$ to obtain the maps $d_y$ as above.  In particular, $d_y$ now has the monotonicity property: if 
$d_y(x,z)\geq\Theta$, then $d_w(x,y),d_w(y,z)\leq d_w(x,z)$ where $w,x,y,z\in\bbY_i$.  Within $\bbY_i$, the maps $d_y$ 
continue to be symmetric and satisfy the properness property, with $\Theta$ replacing $\theta$.  The same theorem also 
provides a constant $\kappa$ so that, for all pairwise distinct $x,y,w,z\in \bbY_i$, we have
\begin{itemize}
     \item $d^\pi_y-\kappa\leq d_y\leq d^\pi_y$;
     \item $d_y(x,z)-\kappa\leq d_y(x,w)+d_y(z,w)$;
     \item $\min\{d_y(x,z),d_x(y,z)\}\leq\kappa$.
\end{itemize}
We emphasise that the constants $\kappa,\Theta$ have been chosen so that the above properties hold within each $\bbY_i$.

\begin{remark}\label{rem:constants_1}
From the proof of~\cite[Theorem 3.3]{BBF}, we see that we can take $\Theta=4\theta+1$.  Indeed, any choice of 
$\Theta>4\theta$ guarantees all of the properties of $d_y$ that we will need.  We can also take any $\kappa\ge3\theta$, 
by~\cite[Proposition 3.2]{BBF}.
\end{remark}

\subsection{Composite rotating families}\label{subsec:CRF}
We now recall the notion of a composite rotating family.  The main idea to keep 
in mind is that we want $\Gamma_x$ to consist of ``large rotations'' around $x$, where $d^\ang_x$ is thought of as the angle 
at $x$.

\begin{defi}(Composite rotating family)\label{def;CRF}
	Consider a composite projection system  $\bbY_*$ endowed with an action of a group $G$  by isomorphisms, i.e.	$G$ 
acts on $\bbY_*$, preserving the partition $\bbY_*=\bigsqcup_{i=1}^m\bbY_i$, and satisfying $\Act(gy)=g\Act(y)$ for all 
$g\in G$ and $y\in\bbY_*$.  Moreover, suppose that if $x,y,z\in\bbY_*$ are such that $d^\ang_y(x,z)$ is defined, then 
$d^\ang_{gy}(gx,gz)=d^\ang_y(x,z)$ for all $g\in G$.
	
	A \emph{composite rotating family} on  $(\bbY_*,G)$, with rotating control $\Theta_{rot}>0$ 
	is a  family of subgroups $\G_v,v\in \bbY_*$ such that   
	\begin{itemize}
		\item for all $x\in \bbY_*, \G_x <  {\rm Stab}_G(x)$,   is an infinite group;
		\item $\G_x$ acts by \textbf{rotations} around $x$ (i.e. whenever $y=x$ or $y\not\in\Act(x)$, the subgroup 
$\G_x$ fixes $y$ and $d^\ang_y$), with
		\item \textbf{proper isotropy}   (i.e. for all 
$R>0, y \in \Act(x)$, the set  $\{\gamma \in \G_x, d_x^\ang (y, \gamma y) <R\}$ is finite);
		\item for all $g\in G$, and all $x\in \bbY_*$, we have $\G_{gx}= g\G_x g^{-1}$;
		\item if $x\notin \Act(z)$ then $\Gamma_x$ and $\Gamma_z$ commute;
		\item for all $i\leq m$ and for all $x,y, z \in \bbY_i$, if $d_y(x,z)
		\leq \Theta$, then $$d_y(x,gz) \geq \Theta_{rot}$$ for all
		$g\in \G_y\setminus\{1\}$.
	\end{itemize}
	
\end{defi}

\begin{assumption}\label{assump:const_1}
 From now and until the end of the subsection, we fix a composite rotating family, and we use the notation from Definition 
\ref{def;CRF}.  Moreover, we assume that the constants are chosen as in Remark \ref{rem:constants_1}. Finally, we set 
$\Theta_0 =2\Theta + 3\kappa$.
\end{assumption}
Recall the useful transfer lemma, which allows one to reduce to ``transfer'' various configurations to a single coordinate.

\begin{lem} (Transfer Lemma, \cite[Lemma 1.4, Prop. 1.6]{D_PlateSpinning})\label{lem;transfert}
	Let $x\in \bbY_*$. For all $i$, there exists $x^t \in \bbY_i$, such that 
	\begin{itemize}
		\item for all $\gamma_x\in \G_x$,   and all $z \in \bbY_i$, one has $d_z^\ang (\gamma_x x^t, x^t) \leq 
\Theta_0 $, and 
		
		\item for all $y\in \bbY_i$ that  is $x$-active, for all  but finitely many elements $\gamma_x$ of $\G_x$, 
one has $d_y^\ang (x,\gamma_x x^t)\leq \kappa$,
		\item there is $\gamma_x\in \G_x$ so that for all $y\in \bbY_i$ that is $x$-active, we have either $d_y^\ang 
(x,\gamma_x x^t)\leq \kappa$, or $d_y^\ang (x, x^t)\leq \kappa$.
	\end{itemize}
Moreover, if $\Gamma_x$ has a fixed point in $\bbY_i$, we can choose $x^t$ to be 
such a fixed point.
\end{lem}

\begin{rem}
 The reader familiar with the construction of projection complexes from \cite{BBF}  will notice that the first bullet says 
that $\G_x$ has an orbit of diameter at most $1$ in the projection complex of $\bbY_i$. This is in fact the defining 
property 
of $x^t$ in \cite{D_PlateSpinning}, and the reason why the ``moreover'' part holds.
\end{rem}

\begin{rem}\label{rem:either_way}
 Notice that by the first and third bullets, we have $d_y^\ang(x,x^t)\leq \Theta_0+\kappa$ (regardless of which case from the 
third bullet applies).
\end{rem}

\begin{cor}\label{cor:large_multicolor_rotations}
 Let $x,y\in \bbY_*$ so that $x$ is $y$-active, and let $\gamma_y\in\Gamma_y-\{1\}$. Then $d_y^\ang(x,\gamma_yx)\geq 
\Theta_{rot}-2(\Theta_0+2\kappa)$.
\end{cor}

\begin{proof}
 Let $x^t$ be as in Lemma \ref{lem;transfert}. Recall that we have $d_y^\ang(x,x^t)\leq \Theta_0+\kappa$ (Remark 
\ref{rem:either_way}).  By equivariance, $d_y^\ang(\gamma_y x,\gamma_y x^t)\leq \Theta_0+\kappa$, and the conclusion follows 
from the (approximate) triangular inequality for $d_y^\ang$.
\end{proof}

Also, the transfer lemma allows one to transfer properness, from Definition \ref{def;CPS}, which will be useful.

\begin{lem}\label{lem:more_proper}
	For all $i$ and all $x,z \in \bbY_i$  the set   $ I= \{y \in \bbY_*, d^\pi_y(x,z)\geq \theta +2\kappa  \}$ is finite.
\end{lem}

\begin{proof}
	Assume it is infinite, and extract an infinite family of elements of the same coordinate $j$. We use  Lemma 
\ref{lem;transfert} to transfer $x$ and $z$ in $\bbY_j$: for $x^t$ and $z^t$ as in the transfer lemma, we have that for all 
$y\in \bbY_j \cap I$, and  for 
$\phi_x(y)$ either the specific element $ \gamma_x\in \G_x$ from the lemma (third point), or the identity, (and $\phi_z$ 
similarly),  one has $d^\ang_y(\phi_x(y) x^t, \phi_z(y) z^t) >  \theta$.  One may extract an infinite family of elements $y$ 
for which the $\phi_x(y)$ are all equal, and the $\phi_z(y)$ are all equal. This provides two elements $\phi_x x^t$ and  
$\phi_z z^t$ of $\bbY_j$ such that  $\{y \in \bbY_j, d^\pi_y(\phi_x x^t, \phi_z z^t)\geq \theta\}$ is infinite, which is a 
contradiction with properness.
\end{proof}

\begin{notation}\label{notation:kernel}
Given a composite rotating family, let $N\leq G$ be the subgroup generated by all the subgroups $\G_x, x\in  \bbY_*$.
\end{notation}

\section{Composite projection graphs}\label{sec:composite_projection_graph}
We say that $(X,\bbY_*)$ is a $G$-\emph{composite projection graph} if
\begin{enumerate}
	\item $X$ hyperbolic graph and $G$ acts on $X$ by simplicial automorphisms.
	\item  \textbf{Composite projection system:}  $\bbY_*=X^{(0)}$ has the structure of a composite projection system on 
which $G$ acts by isomorphisms.  We let $\theta$ be the constant from Definition~\ref{def;CPS} and let  $\Theta,\kappa$ be 
the constants, depending on $\theta$, from the discussion following that definition.  
	\item \textbf{Bounded geodesic image (BGIT):}  There exists $C$ so that the following holds. For each $x,y,s\in 
\bbY_*$ so that $d_s(x,y)$ is defined and larger than $C$, on any geodesic $[x,y]$ there exists a vertex $w$ with 
$d_X(w,s)=1$.\label{item:BGIT}
\end{enumerate}

Moreover, $(\{\Gamma_s\},\bbY^\tau_*)$ is a \emph{composite rotating family} with constant $\Theta_{rot}$ on the 
$G$-\emph{composite projection graph} $(X,\bbY_*)$ if:
\begin{enumerate}
	\item $\{\Gamma_s\}$ is a composite rotating family on a composite projection system $\bbY^\tau_*\subseteq \bbY_*$, 
with constant $\Theta_{rot}$.
	\item $\Gamma_s$ fixes any $w$ with $d_X(w,s)=1$.
\end{enumerate}

Our main technical statement is about properties of the action of $G$ on $X$ that persist for the action of $G/N$ on $X/N$ 
when the rotations are sufficiently large.

Recall that, given a group $G$ acting on a metric space $X$, the element $g\in G$ is \emph{WPD} (weakly proper 
discontinuous) 
if for every $x_0\in X$ and $r\geq 0$  there exists $n_0$ so that for all $n\geq n_0$ the set $$\{h\in G: 
d_X(x_0,hx_0),d_X(g^nx_0,hg^nx_0)\leq r\}$$ is finite.

Our main goal in this section is to prove Theorem~\ref{thm:quotient_proj_graph}.  The proof refers to various statements 
which are postponed to subsequent sections, so that the high-level strategy is made clear before the technicalities are 
introduced.

\begin{thm}\label{thm:quotient_proj_graph}
Let $(X,\bbY_*)$ be a $G$-composite projection graph.
	If $\Theta_{rot}$ is sufficiently large, in terms of $X,G,\bbY_*,\bbY^\tau_*,C$, then for $N=\langle 
\{\Gamma_s\}\rangle$, we have:
	\begin{itemize}
		\item $X/N$ is hyperbolic.
		\item If the action of $G$ on $X$ has a loxodromic then so does the action of $G/N$ on $X/N$.  If the action 
of $G$ on $X$ has a WPD element, then so does the action of $G/N$ on $X/N$.  If the action of $G$ on $X$ is non-elementary, 
then so is the action of $G/N$ on $X/N$.
	\end{itemize}
\end{thm}

\begin{proof}
We will refer to three facts established in the next section, Corollary~\ref{cor:quotient_hyperbolic}, 
Proposition~\ref{prop:preserve_lox}, and Lemma~\ref{lem:non_elem}.  Let $\theta,\Theta,\kappa$ be the constants 
from above, which depend on the composite projection system, and recall that $C$ is the BGIT constant associated to the 
composite rotating family.  Suppose that $\Theta_{rot}$ satisfies 
$\frac{3}{10}\Theta_{rot}>\frac52\Theta_0+3\kappa+2C$. 

Corollary~\ref{cor:quotient_hyperbolic} implies that $X/N$ is hyperbolic. 

Suppose that $\gamma\in G$ acts loxodromically on $X$.  Suppose, moreover, that there exists $x_0\in X$ such that 
$d_s(x_0,\gamma^nx_0)<\Theta_{rot}/10$ for all $n\in\bbZ$ and all $s$ for which the preceding quantity is defined.  
Then Proposition~\ref{prop:preserve_lox} ensures that the image $\bar\gamma\in G/N$ of $\gamma$ is loxodromic on $X/N$, and, 
moreover, if $\gamma$ acts on $X$ as a WPD isometry, then $\bar\gamma$ acts on $X/N$ as a WPD isometry.  We need to show 
that 
for a suitably large $\Theta_{rot}$, there exists a loxodromic (resp. loxodromic WPD) element $\gamma\in G$ with the desired 
small-projection property from Proposition \ref{prop:preserve_lox}. 

Fix a base vertex $x_0\in X$ and let $\gamma\in G$ be loxodromic on $X$ (we choose it loxodromic WPD if there is such an 
element in $G$).  We now show that $d_s(x_0,\gamma^nx_0)$ is uniformly bounded whenever it is defined. This is ultimately a 
consequence of BGIT, which implies that the ``tails'' of the orbit of $\gamma$ do not affect projection distances very much.

Let $L_\gamma=L_\gamma(x_0)$ be the supremum over 
all $n\in\bbZ$, and over all $s$ for which the quantity is defined, of $d_s(x_0,\gamma^nx_0)$.  We claim that 
$L_\gamma<\infty$.

Indeed, let $\delta$ be the hyperbolicity constant for $X$.  Then there exists $\mu=\mu(\gamma,\delta)$ such that 
each geodesic $[x_0,\gamma^nx_0]$ lies at Hausdorff distance at most $\mu$ from the (quasigeodesic) sequence 
$x_0,\gamma x_0,\ldots,\gamma^nx_0$.  

Fix $s$.  By BGIT, either $d_s(x_0,\gamma^nx_0)\leq C$, or $s$ is adjacent to some vertex of $[x_0,\gamma^nx_0]$.  Hence 
there exist integers $p,q$ with $1\leq p\leq q\leq n$ and $|p-q|$ depending only on $\gamma$, so that $s$ is not adjacent to 
a vertex 
of $[x_0,\gamma^ix_0]$ or $[\gamma^ix_0,\gamma^nx_0]$ unless $p\leq i\leq q$.  Thus $d_s(x_0,\gamma^nx_0)\leq 
d_s(\gamma^px_0,\gamma^qx_0)+2C$, by BGIT and the triangle inequality.  It follows that $d_s(x_0,\gamma^nx_0)\leq 
2C+d_{\gamma^{-p}s}(x_0,\gamma^{q-p}x_0)$.  Now, either  $d_{\gamma^{-p}s}(x_0,\gamma^{q-p}x_0)\leq \theta+2\kappa$, in 
which 
case  $d_s(x_0,\gamma^nx_0)\leq \theta+2\kappa+2C$, or $\gamma^{-p}s$ is one of finitely many elements for which 
$d_{\gamma^{-p}s}(x_0, \gamma^{q-p}x_0)>\theta+2\kappa$, by Lemma~\ref{lem:more_proper}.  Letting $L'_\gamma$ be the maximum 
of  $d_{\gamma^{-p}s}(x_0,\gamma^{q-p}x_0)$ over these finitely many elements gives $d_s(x_0,\gamma^nx_0)\leq L'_\gamma+2C$. 
 Hence $L_\gamma<\infty$ for all $\gamma$.

Now let $L=\inf L_\gamma$, where the infimum is taken over the set of $\gamma\in G$ that are loxodromic on $X$.  Note that 
$L$ depends only on the composite projection system, its associated constants, the BGIT constant $C$, and the $G$--action, 
but not on the choice of rotation subgroups.  Suppose that $\Theta_{rot}>10(L+1)$.  Then, any loxodromic $\gamma$ with 
$L_\gamma\leq L+1$ satisfies the hypothesis of Proposition~\ref{prop:preserve_lox} and thus has image $\bar \gamma$ which is 
loxodromic on $X/N$, and WPD if $\gamma$ is itself WPD. 

Finally, suppose that $\gamma$ is as above and that $\gamma'$ is a loxodromic element that is independent of $\gamma$ and 
has the property that $\bar\gamma'$ is also loxodromic on $X/N$.  (If $G$ contains independent loxodromics 
$\gamma,\gamma_1$, then we can choose $\gamma'$ to be a conjugate of $\gamma$ by a sufficiently high power of $\gamma_1$, 
and see from the above argument that $\bar\gamma'$ is again loxodromic on $X/N$.)

Given such a pair $\gamma,\gamma'$, let $L_{\gamma,\gamma'}$ be the supremum, over all $m,n\in\bbZ$ and all $s$ where 
the following quantity is defined, of $d_s(\gamma^mx_0,(\gamma')^nx_0)$.  We claim that $L_{\gamma,\gamma'}<\infty$.  
Indeed, since $\gamma,\gamma'$ are independent loxodromics, there exists $\mu$ such that we have the following.  For all 
$m,n\in\bbZ$, any geodesic $[\gamma^m x_0,(\gamma')^nx_0]$ lies at Hausdorff distance $\leq\mu$ from 
$$\{\gamma^mx_0,\gamma^{m-1}x_0,\cdots,x_0\}\cup\{x_0,\gamma'x_0,\cdots,(\gamma')^nx_0\}.$$  We can now argue as 
above, using BGIT and replacing $\{x_0,\gamma x_0,\ldots,\gamma^nx_0\}$ by 
$\{\gamma^mx_0,\gamma^{m-1}x_0,\cdots,x_0\}\cup\{x_0,\gamma'x_0,\cdots,(\gamma')^nx_0\}$.

Now, letting $\gamma,\gamma'$ vary over all pairs of independent loxodromic elements of $G$ such that 
$\bar\gamma,\bar\gamma'\in G/N$ are loxodromic on $X/N$, take $M=\inf_{\gamma,\gamma'}L_{\gamma,\gamma'}$.  Suppose that 
$\Theta_{rot}>10(M+1)$.  Then any pair $\gamma,\gamma'\in G$ of independent loxodromics such that $\bar\gamma,\bar\gamma'$ 
are loxodromic and $L_{\gamma,\gamma'}\leq M+1$ has the property that $\bar\gamma,\bar\gamma'$ are independent loxodromics, 
by Lemma~\ref{lem:non_elem}.  Hence, if $\Theta_{rot}>10(M+1)$, the action of $G/N$ on $X/N$ is non-elementary.
\end{proof}

\section{Shortenings and their applications}\label{sec:shortening}
\newcommand{\Thetashortenable}{\Theta_{short}}
We work in the setting of Section~\ref{sec:composite_projection_graph}, keeping all notation.  

The results of this section support the lifting procedure developed in Section~\ref{sec:lift_and_project}, which is vital 
for proving the statements (Proposition~\ref{prop:preserve_lox}, Corollary~\ref{cor:quotient_hyperbolic}, 
Lemma~\ref{lem:non_elem}) used in the proof of Theorem~\ref{thm:quotient_proj_graph}.  

The main statements are Corollary~\ref{cor:short_less_complex} and Proposition~\ref{prop:shortening}, on which the corollary 
depends.  In fact, the reader interested in the proof of Theorem~\ref{thm:quotient_proj_graph} is advised to read the 
statement of Corollary~\ref{cor:short_less_complex} and then proceed to Section~\ref{sec:lift_and_project}.  In order to 
understand the statement of Corollary~\ref{cor:short_less_complex}, one needs to know the following.  For each $\gamma\in 
N\setminus\{1\}$, there is an associated complexity $(\alpha(\gamma),n(\gamma))$, where $\alpha(\gamma)$ is a countable 
ordinal and $n(\gamma)\in\bbN$, and $\Thetashortenable$ is a constant depending on $\Theta_{rot}$, $\Theta_0$, and $\kappa$.  
(This value is one of the sources of the ``sufficiently large'' constraint on $\Theta_{rot}$ in 
Theorem~\ref{thm:quotient_proj_graph}.)

\subsection{Structure of the kernel $N$ and complexity of elements}\label{subsec:GOG}
The aim of \cite{D_PlateSpinning} was to investigate the structure of $N$. We may extract the following statement, combining 
the construction from \cite[\S 2.4.2]{D_PlateSpinning} with \cite[Lem 2.16, Prop. 2.13, Lem 2.17]{D_PlateSpinning}:

\begin{theo}\label{theo;statement_that_should_have_been_in_D_PS}
For any countable ordinal $\alpha$ there exists a subset $\bbY_{*_\alpha}$  of   $\bbY_*$ such that,  denoting by 
$N_\alpha$ the subgroup of $N$ generated by $\G_v, v\in   \bbY_{*_\alpha}$,  we have:
	\begin{enumerate}
		\item $N_0=\{1\}$ and $\bbY_{*_0}=\emptyset$
		\item if $\alpha$ is not a limit ordinal, there exists $i(\alpha)$, and a subset $R_{\alpha} \subset 
\bbY_{i(\alpha)}$,   such that  $N_\alpha$ is an amalgamated free product of $N_{\alpha-1}$ with the groups
		$$\G_{v} \times  \langle \G_w, \, w\in  \left(\bbY_{*_{\alpha-1}}   \setminus  \Act(v)\right)  \rangle $$ for 
 $v\in  R_{\alpha}$.  (\cite[Lem. 2.16]{D_PlateSpinning}) 
		
		Notice that each element in the $N_\alpha$-orbit of $R_{\alpha}$ naturally corresponds to a vertex in the 
Bass-Serre tree of the  previous decomposition;  from now on we implicitly identify any such element with the corresponding 
vertex.
		\item \label{item:shortenings_to_alpha-1}
		Suppose that $\alpha$ is not a limit ordinal, and let $T_\alpha$ be the Bass-Serre tree of the  previous 
decomposition.  Also, let $v_1, v_2, v_3$ be three vertices  of $T_\alpha$ in the $N_\alpha$-orbit of the vertices of   
$R_{\alpha}$, with $v_2\in [v_1, v_3]$ in $T_\alpha$. Then, when seen as elements in  $\bbY_{i(\alpha)}$, one has 
$d_{v_2}^\ang(v_1, v_3) \geq \Theta_{rot}- \Theta_0$, and there exists $\gamma_{v_2} \in \G_{v_2}$ such that  
$d_{v_2}^\ang(v_1, \gamma_{v_2} v_3) \leq  \Theta_0$.
		(\cite[Lem. 2.16 with Prop 2.13]{D_PlateSpinning})
				
		\item  if $\alpha$ is a limit ordinal,  $N_\alpha$ is the direct union of $N_\beta, \beta <\alpha$.  
(\cite[Lem. 2.17]{D_PlateSpinning})
		\item   $\bigcup_\alpha \bbY_{*_\alpha} =  \bbY_*$. (\cite[Lem. 2.19]{D_PlateSpinning})
	\end{enumerate}
\end{theo}

We can now define the complexity $(\alpha(\gamma),n(\gamma))$ of $\gamma\in N\setminus 1$.

\begin{defi}\label{defn:alpha}
	Given $\gamma \in N$, let $\alpha(\gamma)$ be the smallest ordinal $\alpha$ for which $\gamma $ is in a conjugate 
of $N_\alpha$ in $N$.  
  Observe that $\alpha(\gamma)$ is never a limit ordinal, by the third point of Theorem 
\ref{theo;statement_that_should_have_been_in_D_PS},  and if $\alpha(\gamma)=0$ then $\gamma=1$. 
\end{defi}

\begin{defi}\label{defn:n}
	Given $\gamma \in N\setminus\{1\}$, consider the amalgamated free product decomposition of a conjugate of 
$N_{\alpha(\gamma)}$ 	containing $\gamma$, given by the second point of Theorem 
\ref{theo;statement_that_should_have_been_in_D_PS}. 

Consider the cyclic normal form of the conjugacy class $[\gamma]$, which 
is either an element of $\Gamma_v$ for some $v\in R_{\alpha(\gamma)} $, or a cyclic word $w=\ell_1 \ell_2 \dots 
\ell_{2r}$, where, for all $i$,  we have $\ell_{2i}\in  N_{\alpha(\gamma)-1}\setminus\{1\}$ and  $\ell_{2i+1} \in 
\Gamma_v\setminus\{1\}$ for some $v\in  R_{\alpha(\gamma)}$.   Let $n(\gamma)$ be the length of this 
cyclic normal form, namely    $n(\gamma)=1$ if $\gamma$ is conjugate into some         $\Gamma_v$ for some $v\in 
R_{\alpha(\gamma)} $, and it is  $n(\gamma)=2r$ for the above $r$ otherwise.
\end{defi}

\begin{rem}
Note that if $\gamma \in N_{\alpha(\gamma)}$ is seen as an element of  $N_{\alpha(\gamma) +1}$, 
the length of its cyclic normal form is $1$, but we do not set this in the notation $n(\gamma)$ since this notation is 
reserved to the amalgam decomposition of $N_{\alpha(\gamma)}$. This way, no $n(\gamma)$ has been multiply defined. We 
adopt the convention that $n(1) = 0$. 
\end{rem}

\subsection{Angles and shortenings}\label{subsec:shortenings}
The main point of the following proposition is to relate the normal form of $\gamma$ with vertices at which one sees a large 
projection between $\gamma$-translates.  The move that will allow us to shorten normal forms can be pictorially described as 
follows: Consider the axis of $\gamma$, a vertex $v$ on the axis, and an element $\gamma_v$ that stabilises $v$ and rotates 
an edge on the axis containing $v$ to the other such edge. Then $\gamma_v\gamma$ has shorter normal form than $\gamma$.

\begin{prop} [Angles and shortenings]\label{prop:shortening}
	Let $\gamma$ be an element of $N_{\alpha(\gamma)}$ with $n(\gamma)>1$.  Then all of the following hold:
	\begin{itemize}
		\item The element $\gamma$ is hyperbolic in the tree $T_{\alpha(\gamma)}$, and its axis in this tree 
contains some vertex $v_0$ in the $N_{\alpha(\gamma)}$-orbit of $R_{\alpha(\gamma)}$. 
		\item For all $v$ in the axis of $\gamma$ and in the $N_{\alpha(\gamma)}$-orbit of $R_{\alpha(\gamma)}$, 
		there exists $\gamma_v\in \G_v \setminus \{1\}$ such that the cyclic normal form of $[\gamma_v \gamma]$ is 
strictly shorter than that of $[\gamma]$. 
		\item If $w$ is a vertex in the tree $T_{\alpha(\gamma)}$,  then
		there exists a  vertex $v$  in the $N_{\alpha(\gamma)}$-orbit of $R_{\alpha(\gamma)}$ such that $v$ lies in 
the intersection of the interior of the segment $[w,\gamma w]$ of $T_{\alpha(\gamma)}$, and of the axis of $\gamma$.   
		Moreover, for all such  $v$, one has $d^\ang_v (w, \gamma w) > \Theta_{rot} - \Theta_0$. 
		\item For all  $\nu \in \bbY_{i(\gamma)}$,  there is a vertex $v$ in the $N_{\alpha(\gamma)}$-orbit of 
$R_{\alpha(\gamma)}$ on the axis of $\gamma$ such that         $d^\ang_v (\nu, \gamma \nu) >(\Theta_{rot} - 
\Theta_0)/2-\kappa$.
		\item Suppose that $j\neq i(\gamma)$ and let $\nu'\in\bbY_j$.   Then either there exists $v$ in the 
$N_{\alpha(\gamma)}$-orbit of $R_{\alpha(\gamma)}$ that is $\nu'$-inactive, and $\gamma_v\in\G_v$ so that $\gamma_v\gamma$ 
has shorter cyclic normal form than $\gamma$, or there exists $v$  that is active for $\nu'$ and $\gamma \nu'$, with the 
property that $d^\ang_v (\nu', \gamma \nu') >(\Theta_{rot} - \Theta_0)/2-2\Theta_0 -3\kappa$.		
	\end{itemize}
\end{prop}

\begin{proof}
	The first point is  a general fact for elements in amalgamated free products. 
	This is also true of the second point, in the specific setting where there is only one orbit of edges 
around $v$ under the action of its stabilizer, as is the case in our situation.   The first part of the third point is also 
general. The distance estimate of the third point follows from   Theorem \ref{theo;statement_that_should_have_been_in_D_PS} 
(\ref{item:shortenings_to_alpha-1}). 
	 
	For the fourth point, we can assume that $\nu$ is not a vertex of $T_{\alpha(\gamma)}$, for otherwise we can just 
use 
the third bullet.  Consider the bi-infinite sequence of points $(\gamma^n v)_{n\in \mathbb{Z}}$ for some vertex $v$ in the 
$N_{\alpha(\gamma)}$-orbit of $R_{\alpha(\gamma)}$ on the axis of $\gamma$. This sequence can be 
thought of as ranging over points of the axis of $\gamma$ in $T_{\alpha(\gamma)}$ or over points of $\bbY_{i(\gamma)}$.   
Notice that $d^\ang_{\gamma^nv}(\gamma^{n-1}v,\gamma^{n+1}v)\geq \Theta_{rot} - \Theta_0$ for every $n$, by the third point 
of Theorem \ref{theo;statement_that_should_have_been_in_D_PS}. One can then see, using 
the Behrstock inequality and induction, that the set $\{n\in \bbZ: d^\ang_{\gamma^n v} (\gamma^{n-1} 
v, \nu) \leq \kappa\}$ is an open interval $(n_0, +\infty )$, for $n_0 \in \mathbb{Z}\cup\{-\infty\}$ (notice that 
$d^\ang_{\gamma^n v} (\gamma^{n-1} 
v, \nu)$ is defined for each $n$). Moreover, $n_0$ 
cannot be equal to $ -\infty$, since in that case, for all negative $n$,  $\gamma^n v$ would be in 
$\bbY^i_{\Theta_{rot}-\Theta_0-\kappa} (\nu, \gamma v )$, contradicting the properness of the projection system. Thus,   
$d^\ang_{\gamma^{n_0} v} (\gamma^{n_0-1} v, \nu) > \kappa$, while 
	$d^\ang_{\gamma^{n_0+1} v} (\gamma^{n_0} v, \nu) \leq \kappa$. The Behrstock inequality again ensures that   
$d^\ang_{\gamma^{n_0-1} v} (\gamma^{n_0} v, \nu) \leq \kappa$.    After translation by 
$\gamma$, we have $  d^\ang_{\gamma^{n_0} v} (\gamma^{n_0+1} v, \gamma \nu) \leq \kappa $.  

There are two cases to consider. First, suppose that $d^\ang_{\gamma^{n_0} v} (\gamma^{n_0-1} v, \nu) \leq (\Theta_{rot} - 
\Theta_0)/2$. In that case, the triangle inequality gives $   d_{\gamma^{n_0} v}^\ang ( \nu, \gamma \nu ) \geq  
(\Theta_{rot} - \Theta_0)/2 -\kappa$ (we used $d^\ang_{\gamma^nv}(\gamma^{n-1}v,\gamma^{n+1}v)\geq \Theta_{rot} - 
\Theta_0$). 
  The second case is when $d^\ang_{\gamma^{n_0} v} (\gamma^{n_0-1} v, \nu) > 
(\Theta_{rot} - \Theta_0)/2$. Then $d_{\gamma^{n_0} v}^\ang ( \nu, \gamma^{-1} \nu ) \geq  (\Theta_{rot} - \Theta_0)/2 
-\kappa$ (we used $d^\ang_{\gamma^{n_0} v} (\gamma^{n_0-1} v, \gamma^{-1}\nu)=d^\ang_{\gamma^{n_0+1} v} (\gamma^{n_0} v, 
\nu) 
\leq \kappa$).   In both cases, we obtained the desired conclusion.
	
	Let us prove the fifth point. The first case is when $\G_{\nu'}$ has no fixed point in $\bbY_{i(\gamma)}$.  In that 
case,   the conjugate $\gamma \G_{\nu'} \gamma^{-1}$ has no fixed point either, since a fixed point for one would give a 
fixed 
point for the other by translation by $\gamma^{-1}$. Then for all $v\in \bbY_{i(\gamma)}$ (so in particular in the axis of 
$\gamma$ in the tree $T_{\alpha(\gamma)}$), $\nu'$ and $\gamma \nu'$ are $v$-active. We now consider $\nu'^t$ as in the 
transfer lemma, and
we apply the previous point for $\nu'^t$. There is $v$ on the axis of $\gamma$ such that $d^\ang_v (\nu'^t, \gamma 
\nu'^t) >(\Theta_{rot} - \Theta_0)/2-\kappa$. By Remark \ref{rem:either_way}, we have $d^\ang_v (\nu', \gamma \nu'^t) 
>(\Theta_{rot} - \Theta_0)/2-2\kappa -\Theta_0$. 

Notice that $d^\ang_v(\gamma\nu',\gamma\nu'^t)=d^\ang_{\gamma^{-1}v}(\nu',\nu'^t)$,  and the latter quantity is again 
bounded 
by $\Theta_0+\kappa$ by Remark \ref{rem:either_way}. Hence, we get $$d^\ang_v (\nu', \gamma \nu') >(\Theta_{rot} - 
 \Theta_0)/2-3\kappa -2\Theta_0.$$

 Let us now treat the case where $\G_{\nu'}$ has a fixed point in  $\bbY_{i(\gamma)}$. 
Pick one fixed point $\nu'^t$ and consider first  the case where $\gamma \nu'$ is not active for one of the vertices 
of the axis. Then, by the second bullet, one could shorten the length of $[\gamma]$ using the element $\gamma_v$ associated 
to this vertex in 
the normal form of $[\gamma]$, so we are done.
We may thus assume that $\gamma \nu'$ is $v$ active for all $v$ in the axis. It follows that $\nu'$ is 
active for all $v$ in the axis as well, and the argument of the previous case can be applied (in view of the ``moreover'' 
part of the transfer lemma).
\end{proof}

\subsection{Rotations to reduce complexity}\label{subsec:lower_complexity}
For convenience, let $\Thetashortenable =  (\Theta_{rot} -\Theta_0)/2 -2\Theta_0 -3\kappa$.

\begin{cor}[Rotating to reduce the complexity $(\alpha,n)$]\label{cor:short_less_complex}
	For all $\gamma\in N\setminus\{1\}$, and all $x\in \bbY_*$, there is $(s,\gamma_s)$ (here $s\in\bbY_*$ and 
$\gamma_s\in\G_s$)  so that  
$(\alpha(\gamma_s\gamma),n(\gamma_s\gamma))< (\alpha(\gamma),n(\gamma))$ in lexicographic order and either
	\begin{enumerate}
	 \item $x$ is $s$-inactive, or
	 \item $x$ and $\gamma x$ are $s$-active and $d^\ang_s(x,\gamma x)> \Thetashortenable$.	
	\end{enumerate}
\end{cor}

\begin{proof} 
After conjugation by a suitable element $h$ and replacing $x$ with $h^{-1}x$, one can assume that $\gamma \in 
N_{\alpha(\gamma)}$. If $x\in \bbY_{i(\gamma)}$ then  we consider $v$ as in the fourth bullet from 
Proposition~\ref{prop:shortening} and set $s=v$. The cyclic 
normal form of $\gamma_s\gamma$ is shorter than that of $\gamma$ by the second bullet from Proposition~\ref{prop:shortening}. 
Otherwise, if $x\in \bbY_{j}$ with 
$j\neq i(\gamma)$, then by the fifth bullet from Proposition~\ref{prop:shortening} we either proceed as above, or we find 
some $x$-inactive $s$ so that $\gamma_s\gamma$ has shorter cyclic normal form.

In either case, the cyclic length of the conjugacy class is reduced. Either the result is still 
greater than $2$, and in that case $\alpha(\gamma_s\gamma)=\alpha(\gamma)$,  or it is reduced to $1$ (or $0$)  and $\gamma_s 
\gamma$ is actually conjugate into   $N_{\alpha-1}$. In the latter case, one has $\alpha(\gamma_s\gamma)<\alpha(\gamma)$.  
Thus $(\alpha(\gamma_s\gamma),n(\gamma_s\gamma))< (\alpha(\gamma),n(\gamma))$.
\end{proof}
Note that in the last case of the proof (in which $\alpha(\gamma_s\gamma)<\alpha(\gamma)$), the value of $n(\gamma_s\gamma)$ 
can be arbitrary.

\section{Lifting and projecting}\label{sec:lift_and_project} 
In this section, we describe how to lift quadrilaterals and triangles in $X/N$ to $X$.  This will allow us to prove the 
various statements referenced in the proof of Theorem~\ref{thm:quotient_proj_graph}.

\begin{assumption}\label{assump:const}
 We fix the notation of Theorem \ref{thm:quotient_proj_graph}, and fix constants 
as in Standing assumption \ref{assump:const_1}. Let $p:X\to X/N$ be the quotient 
map.  We assume that $3\Theta_{rot}/10>2C+3\kappa+5\Theta_0/2$.
\end{assumption}

\subsection{Lifting}
First, paths (and, more particularly, geodesics) lift:

\begin{lemma}
For each combinatorial path $\bar\gamma$ in $X/N$ starting at $\bar x$, and any point $x$ in the preimage of $\bar 
x$ (henceforth: a \emph{lift} of $x$),  there exists a combinatorial path in $X$ so that $p\circ\gamma=\overline\gamma$, 
which we call a \emph{lift} of $\bar\gamma$. Moreover, if $\bar\gamma$ is a geodesic, then so is $\gamma$.
\end{lemma}

\begin{proof}
In order to lift combinatorial paths it suffices to lift edges, given a lift of the starting point of the edge in 
the quotient.  This can be done since the action of $N$ on $X$ is simplicial. Lifting a geodesic yields a geodesic because 
the quotient map is 1-Lipschitz.
\end{proof}

The following proposition is the key to our approach to study $X/N$ and $G/N$: it allows us to lift geodesic triangles and 
quadrilaterals in $X/N$ to triangles/quadrilaterals in $G/N$, thereby allowing us to translate properties of $X$ (e.g. 
hyperbolicity) to properties of $X/N$.  By requiring the constant $\Theta_{rot}$ to 
be even larger, we could ensure that we can lift $n$-gons for any given $n$, but we will only need the cases $n=3$ and 
$n=4$.  The ``moreover'' part will only be needed for the WPD property.

\begin{prop}\label{prop:lift}
	For each geodesic quadrangle $\overline Q$ in $X/N$ there exists a geodesic quadrangle $Q$ in $X$ so that 
$p(Q)=\overline Q$. We call $Q$ a lift.
	
	Moreover, if the geodesics $[\bar v_1,\bar w_1]$, $[\bar v_2,\bar w_2]$ of $Q$ have lifts $[v_i,w_i]$ so that 
$d_s(v_i,w_i)\leq \Theta_{rot}/10$ whenever the quantity is defined,  then the lift $[v'_i,w'_i]$ of $[\bar v_i,\bar w_i]$ 
contained in $Q$ is an $N$-translate of $[v_i,w_i]$.
\end{prop}

\begin{proof}
	Let $\bar x, \bar y,\bar z, \bar t$  be the vertices of a quadrangle in
	$X/N$. Lift $\bar x$ as $x\in X$, and lift all four geodesic segments
	to get geodesics $[x,y], [y,z], [z,t], [t,x']$. In the setting of the ``moreover'' part, choose an $N$-translate of 
$[v_i,w_i]$ as the lift of $[\bar v_i,\bar w_i]$. There is an element
	$\gamma \in N$ such that $\gamma x = x'$. We argue, by transfinite induction on
	the pair $(\alpha(\gamma), n(\gamma))$ (for lexicographic order). 
	
	If  $(\alpha(\gamma), n(\gamma)) = (0, n)$,  then $\gamma=1$ and $x'=x$, so we are done. 
	We thus assume that  $(\alpha(\gamma), n(\gamma)) = (\alpha, n)$, with $\alpha>0$.
	
	Let us consider $(s,\gamma_s)$ as in Corollary \ref{cor:short_less_complex}. 
	If $x$ is $s$-inactive, we have $\gamma_sx=x$. We can then apply $\gamma_s$ to all lifts, and conclude by induction 
hypothesis.

Otherwise, we have $d_s(x, 
\gamma x) > \Thetashortenable$.  Viewing $s$ as a vertex of $X$, and using assumption~\eqref{item:BGIT} (BGIT) of 
the definition of composite projection graph, we see that in $X$ the 
geodesic $[x,x']$ contains a point at distance $1$ from $s$, and fixed by $\Gamma_s$. 
	
	There are several cases. The first one is when all points $x,y,z,t,x'$ are active for $s$. Recall that $d_s(x,x') >  
\Thetashortenable$, and  by the triangle inequality and our assumption on $\Theta_{rot}$, at least one  of  $ d_s(x,y), 
d_s(y,z), d_s(z,t)$ or   $d_s(t,x') $ is larger than the constant $C$ from BGIT (and, in the setting of the moreover part, 
the pair giving a large projection can be chosen not in $\{(v_i,w_i)\}$). 
	
	Let $(u,u')$ (not in $\{(v_i,w_i)\}$) be a pair among the aforementioned ones such that $d_s(u,u') >C$. 
	By BGIT, it follows that $[u,u']$ contains a point at distance $\leq 1$ from $s$, hence fixed by $\gamma_s$. 
Replacing $u'$ and all the points in $(x,y,z,t,x')$ after the position of $u'$ by their image by $\gamma_s$ produces new 
isometric lifts of the segments, in such a way that the endpoint $\gamma_s x'$ now differs from $x$ by $\gamma_s\gamma$, 
see Figure \ref{fig:rotate_lift}.
	Also, in the setting of the moreover part, notice that we replaced each $[v_i,w_i]$ by an $N$-translate. The 
induction hypothesis allows us to conclude.
	
	\begin{figure}[h]
	\centering
	 \includegraphics[width=0.8\textwidth]{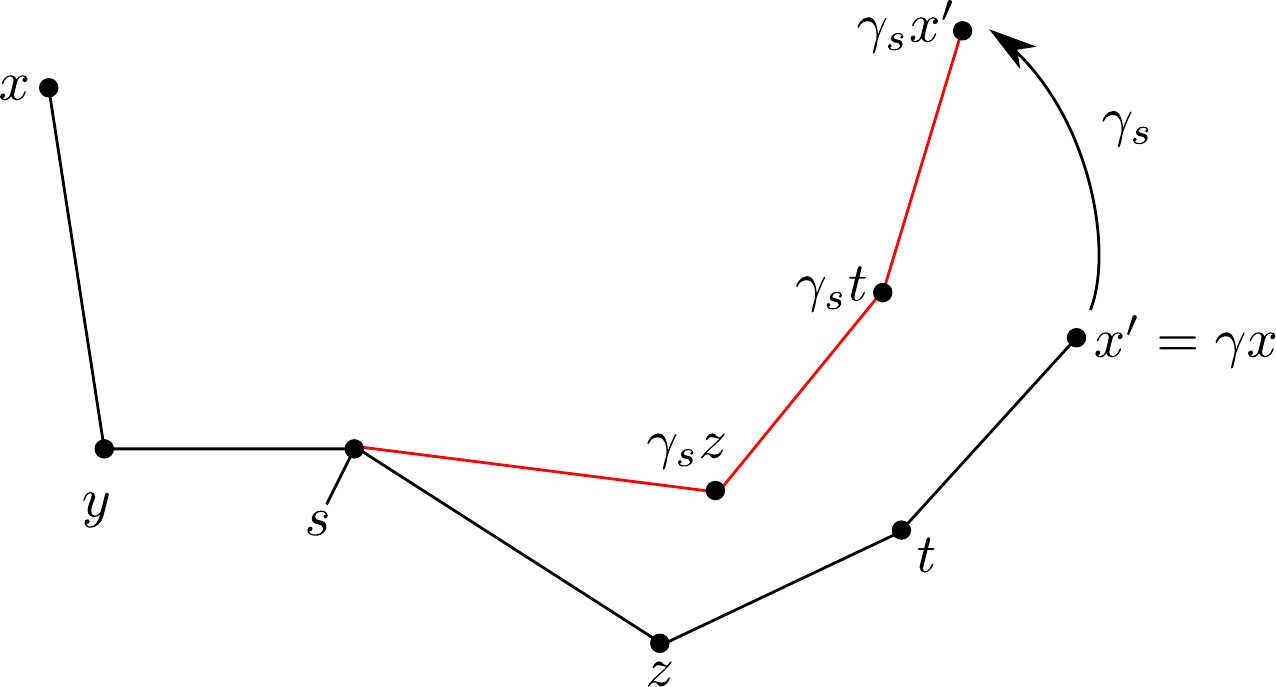}
	 \caption{The case that $s$ is adjacent to a vertex on $[y,z]$. In all cases, $s$ is adjacent to some vertex on one 
of the sides (possibly one of $x,y,z,t,x'$).}\label{fig:rotate_lift}
	\end{figure}
	
	The second case is when not all  $x,y,z,t,x'$ are active for $s$. The argument is similar. Let $u$ be the first point 
in the tuple that is inactive for $s$. This implies that $\gamma_s$ fixes $u$. Changing the lifts of all the elements after 
$u$ by their images by $\gamma_s$ does not change the property that the lifts of segments are isometric. The conclusion is 
the same.
\end{proof}

\subsection{Projecting}
The following lemma says, informally, that directions with small projection angles in $X$ are preserved by $p$.

\begin{lemma}\label{lift;segments_witout_angles}
Suppose that $x,y\in X^{(0)}$ have the property that $d_s(x,y)<\Theta_{rot}/10$ whenever the quantity is defined. 
Then $p|_{[x,y]}$ is isometric, for any geodesic $[x,y]$.
\end{lemma}

\begin{proof}
Assume there is a shorter path from $\bar x=p(x)$ to $\bar y=p(y)$, for some $x,y$ as in the statement.  Lift this 
shorter path as a geodesic segment $[y,x']$. There exists $\gamma \in N$ such that $\gamma x = x'$. Again we proceed by 
induction on $(\alpha(\gamma), n(\gamma))$, this time to prove that  $d_X(y,x') = d_X(y,x)$  (thus falsifying that $d_{X/N} 
(\bar x,  \bar y) <  d_{X} (x,y)$), for any pair $x,y$ as in the statement. If $\alpha(\gamma) = 0$ then $\gamma= 1$ and it 
is obvious.

Assume that  $\alpha(\gamma) >0$. Let $(s,\gamma_s)$ be as in Corollary \ref{cor:short_less_complex}.  If $x$ 
is $s$-inactive, then $\gamma_s x=x$, and we can apply $\gamma_s$ to both geodesics and conclude by the induction hypothesis.
	
Otherwise, $d_s(x,\gamma x)>\Thetashortenable$,  and in particular there exists $s'$ on $[x,x']$ 
at distance $1$ from $s$. Since $d_s(x,y)$ is assumed to be small (if defined) we also have that $s' \in [y,x']$ (actually 
this is also true if $d_s(x,y)$ is not defined: in that case, $s'=y$). Thus, one can change the lift of $[\bar y, x]$ as 
$[y, \gamma_s x']$, while keeping it an isometric lift. One concludes by induction hypothesis, which applies to  
$\gamma_s\gamma$.
\end{proof}

\subsection{Properties of $G\curvearrowright X$ persisting for $G/N\curvearrowright X/N$}
We can now prove the various statements referenced in the proof of Theorem~\ref{thm:quotient_proj_graph}, thereby completing 
the proof of that theorem.

First, lifting geodesic triangles from $X/N$ to $X$ using Proposition~\ref{prop:lift}, we deduce:

\begin{cor}\label{cor:quotient_hyperbolic}
	$X/N$ is $\delta$-hyperbolic, where $\delta$ is the hyperbolicity constant of $X$.
\end{cor}

Next, we investigate survival of loxodromic (resp. WPD) elements of $G$ as loxodromic (resp. WPD) elements of $G/N$:

\begin{prop}\label{prop:preserve_lox}
Assume that $\phi\in G$ is a loxodromic isometry of $X$, and that there is $x_0\in X$  such that for all $n\in 
\mathbb{Z}$ for all $s\in X$,  we have $d_s(x_0, \phi^n x_0)<\Theta_{rot}/10$ whenever it is defined. Then, $\bar \phi \in 
G/N$   is loxodromic on $\bar X = X/N$. Moreover, if $\phi$ has the WPD property, then so does $\bar\phi$.
\end{prop}

\begin{proof}
Lemma \ref{lift;segments_witout_angles} guarantees that $d_{\bar X} (\bar x_0, {\bar \phi}^n \bar x_0)$ grows 
linearly, thus $\bar \phi$ is loxodromic.

By the WPD property, we may take $n$ such that $d(x_0, \phi^n x_0  )$ is so large that, for a certain $k_0$,  at 
most $k_0$ elements of $G$ move the pair $(x_0, \phi^n x_0)$ at distance $\leq 100\delta$ from itself.  Call $ x_n=  { 
\phi}^n  x_0$.
	
Assume that $\{\bar \phi_i, i=1, \dots, k\}$ is a collection of distinct elements that move each  of $\bar x_0, \bar 
x_n$ within distance $100\delta$ from itself. Notice that each $\bar\phi_i[\bar x_0,\bar x_n]$ has a lift with small 
projections as in the moreover part of Proposition \ref{prop:lift}, namely $\phi'_i[x_0,x_n]$ for some $\phi'_i$ in the 
preimage of $\bar \phi_i$ (where we choose $[\bar x_0,\bar x_n]=p([x_0,x_n])$). Hence, there is a lift of a quadrilateral  
$(\bar x_0, \bar \phi_i \bar x_0, \bar \phi_i \bar x_n, \bar x_n)$ of the form $(x_0, \phi_i x_0, \phi_i x_n, x_n)$, for 
some $\phi_i$ in the preimage of $\bar \phi_i$. (We took an $N$-translate of the quadrilateral from Proposition 
\ref{prop:lift} to make sure that one side is $[x_0,x_n]$.)
	
By definition of $k_0$,  if $k> k_0$, then there are $i\neq j$ such that $\phi_i= \phi_j$. Projecting in $G/N$,  
$\bar\phi_i= \bar\phi_j$, which is a contradiction. This forces $k\leq k_0$. This holds for all sufficiently large $n$.  By, 
for example, the proof of \cite[Proposition 
5.31]{DGO}, this is sufficient to ensure 
the WPD property for $\bar \phi$.
\end{proof}

Finally, we see that non-elementarity persists:

\begin{lemma}\label{lem:non_elem}
Suppose that there exist independent loxodromic isometries $\phi,\psi\in G$ of $X$, and $x_0\in X$ such that 
for all $n,m\in \mathbb{Z}$ for all $s\in X$,  $d_s(\phi^n x_0, \psi^m x_0)<\Theta_{rot}/10$ whenever it is defined. 
Then the action of $G/N$ on $X/N$ is non-elementary.
\end{lemma}

\begin{proof}
By Proposition \ref{prop:preserve_lox}, $\bar\phi$ and $\bar\psi$ are loxodromic on $X/N$. Moreover, 
by Lemma \ref{lift;segments_witout_angles}, $p$ restricted to $\langle\phi\rangle x_0\cup \langle \psi\rangle x_0$ 
is an isometry, so that $\bar\phi$ and $\bar\psi$ are in fact independent.
\end{proof}

\subsection{Compatibility with stabilisers}
The following proposition is not used in the proof of Theorem \ref{thm:quotient_proj_graph}, but it is useful in 
applications.  It describes the structure of stabilisers of $X/N$, which turn out to be exactly what one expects.

\begin{prop}\label{prop:stab}
Under the Standing Assumption \ref{assump:const}, the following holds.	For any vertex $v$ of $X$,
	$${\rm Stab}(v)\cap N=\langle \{\Gamma_w\cap {\rm Stab}(v): w\in\bbY_*\} \rangle = \langle \{\Gamma_w: w\in\bbY_* 
\setminus \Act(v)\} \rangle.$$
\end{prop}

\begin{proof}
To see the second equality, observe that if $w\notin \Act(v)$ then $\Gamma_w\subset {\rm Stab}(v)$.  If 
$\Gamma_w\cap {\rm Stab}(v)$ is non-trivial, then since $\G_w$ consists only of large rotations,  $w$ must be $v$-inactive 
(in other words $w\notin \Act(v)$). In fact, if $w$ is $v$-active, then, for any nontrivial $\gamma_w\in \Gamma_w$, 
$d_w(v,\gamma_w v)$ is defined, and hence nonzero by Corollary \ref{cor:large_multicolor_rotations}.  Thus $\gamma_w v\neq 
v$.
	
The right-hand-side is contained in  the left-hand-side. Let us prove the other inclusion.  Take $\gamma \in {\rm 
Stab}(v)\cap N$. We want to show that  $\gamma \in   \langle \{\Gamma_w: w\in\bbY_* \setminus 
\Act(v)\} \rangle $.  We proceed by induction on $(\alpha(\gamma), n(\gamma))$. 

If  $\alpha(\gamma) =0$, then $\gamma$ is 
trivial. If $\alpha(\gamma)>0$, let $(s,\gamma_s)$ be as in Corollary \ref{cor:short_less_complex}.  If $s$ is $v$-inactive, 
then $\gamma_s\gamma v= v$, and the induction hypothesis applies to $\gamma_s\gamma$. In particular, $\gamma$ is also in 
$\langle \{\Gamma_w: w\in\bbY_* \setminus \Act(v)\} \rangle$.
	
On the other hand, we cannot have $d_s^\ang(v, \gamma 
v)>\Thetashortenable $,  since this contradicts $v=\gamma v$.
\end{proof}

\section{Applications to mapping class groups}\label{sec:MCG}
Let $\Sigma_{g,p}$ denote the genus--$g$ oriented surface with $p$ punctures, and let $\mathcal C(\Sigma_{g,p})$ denote its 
curve graph.

\begin{prop}\label{prop:mcg_cpg}
 There exists $K$ such that the following holds: $\mathcal C(\Sigma_{g,p})$ is an $MCG(\Sigma_{g,p})$-composite projection 
graph.  For any $k\in\mathbb Z\backslash\{0\}$, there is a composite rotating family on the above composite projection 
graph such that, for each curve $y$, we have $\Gamma_y=\langle T_y^{kK}\rangle$, where $T_y$ is the Dehn twist about $y$.
\end{prop}

\begin{proof}
This follows the discussion in~\cite[Section 3]{D_PlateSpinning} exactly, so we just describe the data here and refer 
the reader to~\cite{D_PlateSpinning} for the explanation of why this data determines a composite projection 
graph and rotating family.

\textbf{Composite projection graph:}  $\mathcal C(\Sigma_{g,p})$ is hyperbolic~\cite{MM_I}, and $MCG(\Sigma_{g,p})$ acts 
by simplicial automorphisms.  For each (isotopy class of) curve $y$, let $\Act(y)$ be the set of curves $x$ that intersect
$y$, i.e. the set of vertices of $\mathcal C(\Sigma_{g,p})$ distinct from, and not adjacent to, $y$.  Note that 
$x\in\Act(y)$ if and only if $y\in\Act(x)$.  The $MCG(\Sigma_{g,p})$--invariant colouring $\mathcal 
C(\Sigma_{g,p})=\bigsqcup_{i=0}^m\mathbb Y_i$ is described in~\cite{D_PlateSpinning} and is derived from the 
colouring in~\cite{BBF}.  (There is a finite-index normal subgroup $G_0\leq MCG(\Sigma_{g,p})$ that preserves each colour, 
and the colours correspond to the cosets of $G_0$, so that $MCG(\Sigma_{g,p})$ acts on the set of colours.)  Given a curve 
$y$, and curves $x,z$ intersecting $y$, the distance $d^\pi_y(x,z)$ is defined via subsurface projection, and satisfies the 
properties from Definition~\ref{def;CPS} by results from \cite{MM_II,Beh:thesis}, see the discussion 
in~\cite{D_PlateSpinning}. 

\textbf{Composite rotating family:}  The discussion in~\cite{D_PlateSpinning} provides an integer $K>0$ such that the 
following holds for all $k\in\mathbb Z\backslash\{0\}$.  For each curve $y$, let $T_y$ be the Dehn twist about $y$.  Let 
$\Gamma_y=\langle T_y^{kK}\rangle$.  Then the subgroups $\Gamma_y,\ y\in\mathcal C(\Sigma_{g,p})^{(0)}$, form a composite 
rotating family on the $MCG(\Sigma_{g,p})$--composite projection system discussed above.
\end{proof}

\begin{thm}\label{thm:mcg_acyl}
 Let $\Sigma_{g,p}$ be a finite-type surface, with either $g>0$ or $p>3$. Then there exists 
 a positive integer $K_0$ so that for all non-zero multiples $K$ of $K_0$, the group $MCG(\Sigma_{g,p})/DT_K$ is 
acylindrically hyperbolic, where $DT_K$ is the normal subgroup generated by all $K^{th}$ powers of Dehn twists. 
\end{thm}

\begin{proof}
  By Theorem \ref{thm:quotient_proj_graph}, the group $MCG(\Sigma_{g,p})/DT_K$ 
acts non-elementarily on $\mathcal C(\Sigma_{g,p})/DT_K$ with loxodromic WPD 
elements.  Hence $MCG(\Sigma_{g,p})/DT_K$ is acylindrically hyperbolic, by~\cite[Theorem 1.2]{Osin:acyl}.
 \end{proof}
 
 \begin{rem}
  A similar theorem also holds for quotients of mapping class groups by powers of Dehn twists around curves of one specified 
topological type (we allow $\bbY^\tau_*$ to be a proper subset of $\bbY_*$ in the definition of a composite rotating family 
on a composite projection graph).
 \end{rem}

\subsection{Relative hyperbolicity and relative quasiconvexity}
We will see that, in low-complexity, quotients of mapping class groups by powers of Dehn twists are in fact hyperbolic 
groups. To prove this, we will use relative hyperbolicity.

We will use the following definition of relative hyperbolicity (following Proposition~4.28 of~\cite{DGO}):

\begin{defi}\label{defi:relhyp}
Let $G$ be a group and let $\mathcal H$ be a collection of subgroups of $G$.   Let $\mathcal 
S\subset G$ be a finite set that is closed under taking inverses.  Suppose that $\mathcal S\sqcup\bigsqcup_{H\in\mathcal H}$ 
generates $G$, and let $Cay(G,\mathcal S\sqcup\bigsqcup_{H\in\mathcal H})$ be the Cayley graph with respect to this 
generating set.

(Note that, if $s\in \mathcal S$ is also contained in some $H\in\mathcal H$, and $g\in G$, we regard $g$ and $gs$ as being 
joined by two edges labelled by $s$, one of which is in a the graph $g\Gamma_H$ defined below.)

For each $H\in\mathcal H$, let $\Gamma_H$ be the Cayley graph of $H$ with respect to the generating set $H$, so that 
$\Gamma_H$ is a diameter--$1$ subgraph of $Cay(G,\mathcal S\sqcup\bigsqcup_{H\in\mathcal H})$ for each $H\in\mathcal H$.  
Define a metric $d_H$ on $H$ as follows: given vertices $x,y\in H$, a combinatorial path $\sigma$ in $\Cay(G,\mathcal 
S\sqcup\bigsqcup_{H\in\mathcal H})$ is \emph{admissible} if $\sigma$ does not traverse an edge of $\Gamma_H$.  Then 
$d_H(x,y)$ is the infimum of the lengths of admissible paths from $x$ to $y$. We say that $G$ is \emph{hyperbolic relative 
to $\mathcal H$} if $Cay(G,\mathcal S\sqcup\bigsqcup_{H\in\mathcal H})$ is a 
hyperbolic graph and the metric $d_H$ is proper for each $H\in\mathcal H$.
\end{defi}

We also require the notion of (strong) relative quasiconvexity.  In fact, we will take as the definition the 
characterisation provided by Theorem~4.13 of~\cite{Osin_relhyp}:

\begin{defi}\label{defi:relquas}
Let $G$ be hyperbolic relative to a collection $\mathcal H$ of subgroups.  The subgroup $Q\leq G$ is 
\emph{strongly relatively quasiconvex} if:
\begin{itemize}
     \item $Q$ is generated by a finite set $Y$;
     \item letting $d_Y$ be the word-metric on $Q$ with respect to $Y$, the inclusion $(Q,d_Y)\to Cay(G,\mathcal 
S\sqcup\bigsqcup_{H\in\mathcal H})$ is a quasi-isometric embedding.
\end{itemize}
\end{defi}

We will also use the following well-known fact (see~\cite[Theorem 5.1]{CharneyCrisp}):
 
 \begin{lem}\label{lem:rel_MS}
 Let $G$ act cocompactly on the graph $X$. Let $\{v_i\}$ be a (necessarily finite) collection of representatives of the 
$G$-orbits of the vertices of $X$.  Then $G$ is finitely generated relative 
to $\{Stab(v_i)\}$, and any orbit map defines a quasi-isometry between $Cay(G,\mathcal S\sqcup\bigsqcup_{v_i}Stab(v_i))$ and 
$X$, for any finite relative generating set $\mathcal S$.
\end{lem}

\subsection{Vertex links in the curve graph}\label{sec:links}

Let $(g,p)\in \{(0,5),(1,2)\}$, and let $\gamma\in\mathcal C(\Sigma_{g,p})^{(0)}$. The link $lk(\gamma)$ of $\gamma$ in 
$\mathcal C(\Sigma_{g,p})$ is a 
discrete set that can be identified with the vertex set of the curve graph of the only component of $\Sigma_{g,p}-\gamma$ 
which is not homeomorphic to $\Sigma_{0,3}$. We equip $lk(\gamma)$ with the metric induced by this 
identification of $lk(\gamma)$ with a subset of the curve graph of the aforementioned complementary component. Denote by 
$L_\gamma$ the metric space with underlying set $lk(\gamma)$, endowed with the metric just described.

For a vertex $x\in\mathcal C(\Sigma_{g,p})$, with $x\neq \gamma$, we denote by $r(x)$ the entrance point in $lk(\gamma)$ of 
any geodesic from $x$ to $\gamma$.

The following is an easy consequence of the Bounded Geodesic Image Theorem, \cite[Theorem 3.1]{MM_II}.

\begin{lemma}\label{lem:bgi_complexity_1}
There exists $C$ so that the following holds. Let $x,y,s$ be curves.  If some geodesic $[x,y]$ in 
$\mathcal C(\Sigma_{g,p})$ from $x$ to $y$ does not contain $s$, then $d_{L_s}(r(x),r(y))\leq C.$
\end{lemma}

Roughly speaking, the lemma says that there is a Lipschitz retraction of the complement of $\gamma$ onto $lk(\gamma)$. We 
will use such retraction to prove that the relative metric is finite by, again roughly speaking, starting with a path in the 
complement of $\gamma$ and constructing a path in $lk(\gamma)$ which is not much longer.

 \subsection{Hyperbolic and relatively hyperbolic quotients}\label{subsec:rel_hyp_quotient}
 Fix $(g,p)\in\{(0,4),(0,5),(1,0),(1,1),(1,2)\}$.  Let $K\in\mathbb Z\setminus\{0\}$, and let $DT_K$ be 
the normal subgroup generated by all $K^{th}$ powers of Dehn twists. Let $G(g,p,K)=MCG(\Sigma_{g,p})/DT_K$, and let 
$X_K=\mathcal C(\Sigma_{g,p})/DT_K$.

\begin{thm}\label{thm:mcg_quotients}
Let $(g,p)$ be as above.  Then there exists a positive integer $K_0$ so that for all sufficiently large multiples $K$ of 
$K_0$, the following hold. 
 \begin{enumerate}
  \item Suppose $(g,p)\in \{(0,4),(1,0),(1,1)\}$. Then $G(g,p,K)$ is non-elementary hyperbolic.\label{item:small}

  \item Suppose $(g,p)=(0,5)$ (resp. $(g,p)=(1,2)$). Then $G(g,p,K)$ is hyperbolic relative to an infinite index subgroup 
$H$ 
virtually isomorphic to $G(0,4,K)$ (resp. relative to two infinite index subgroups $H_1,H_2$, one virtually isomorphic to 
$G(0,4,K)$ and one virtually isomorphic to $G(1,1,K)$).  In particular, $G(g,p,K)$ is non-elementary 
hyperbolic.\label{item:medium}
 \end{enumerate}
  Moreover, letting $\mathcal H$ be the set of peripheral subgroups arising in the second case, we have that for each 
relative generating set $\mathcal S$, there is a $G_K$--equivariant 
quasi-isometry $Cay(G_K,\mathcal S\cup\bigcup_{H\in\mathcal H}H)\to X_K$.
\end{thm}

\begin{proof}
The curve graph $\mathcal C(\Sigma_{g,p})$ is an 
$MCG(\Sigma_{g,p})$-composite projection graph, and large powers 
of Dehn twists define a composite 
rotating family on it.  Thus, there exists $K_0$ so that for sufficiently large $K\in 
K_0\mathbb Z$, the graph $X_K$ is 
hyperbolic by Theorem \ref{thm:quotient_proj_graph}.  Moreover, the action of $G(g,p,K)$ on $X_K$ is 
non-elementary, since the action of $MCG(\Sigma_{g,p})$ on $\mathcal C(\Sigma_{g,p})$ is non-elementary. In particular, 
$G(g,p,K)$ contains elements acting on $X_K$ loxodromically.

 \textbf{Hyperbolicity in lowest-complexity cases:}  Suppose that $(g,p)\in 
\{(0,4),(1,0),(1,1)\}$. Then the action of $G(g,p,K)$ on $X_K$ has finite vertex 
stabilisers.  Indeed, each vertex stabiliser is a quotient of a vertex 
stabiliser of $\mathcal C(\Sigma_{g,p})$, which in our 
case is virtually generated by a Dehn twist.  By Lemma~\ref{lem:rel_MS} and cocompactness of the $G(g,p,K)$--action on 
$X_K$, $G(g,p,K)$ is $G(g,p,K)$--equivariantly quasi-isometric to $X_K$ and is thus hyperbolic.  (Since $MCG(\Sigma_{g,p})$ 
is hyperbolic in this case, hyperbolicity of $G(g,p,K)$ can also be deduced from \cite{Del:kill_power}.)
 
 \textbf{Stabilisers:} Suppose from now on that $(g,p)$ is as in the second case. By Proposition \ref{prop:stab}, all 
infinite vertex stabilisers for the action of $G(g,p,K)$ on $X_K$ are of the form specified by 
the statement. Specifically, for $(g,p)=(0,5)$ there is only one topological type of curves (yielding exactly one conjugacy 
class of stabilisers in $X_K$), with stabiliser virtually isomorphic to a central extension of $MCG(\Sigma_{0,4})$ by a 
cyclic subgroup generated by a Dehn twist. Proposition \ref{prop:stab} guarantees that the image of such a stabiliser in 
$G(g,p,K)$ is also obtained by modding out powers of Dehn twists. 

For $(g,p)=(1,2)$ the situation is similar, except that there are two topological types of curves (one 
non-separating, with complement $\Sigma_{0,4}$, and one separating with 
complement $\Sigma_{0,3}\sqcup \Sigma_{1,1}$). 
 
 \textbf{Relative hyperbolicity:}  By Lemma \ref{lem:rel_MS}, $Cay(G(g,p,K),\mathcal S\cup 
\mathcal H)$ is equivariantly quasi-isometric to $X_K$, where $\mathcal S$ is any fixed finite generating set, and 
$\mathcal H$ is a union of conjugacy representatives of stabilisers. Since the action of $G(g,p,K)$ on $X_K$ has a 
loxodromic element, the stabilisers must have infinite index. We now have to prove that the relative metric on each 
stabiliser, $H$, is proper.

Recall from Subsection \ref{sec:links} that, for $\gamma$ a curve on $\Sigma_{g,p}$, we defined a metric space $L_\gamma$ 
with underlying set the link $lk(\gamma)$ of $\gamma$ (regarded as a vertex of $\mathcal C(\Sigma_{g,p})$). Moreover, 
$L_\gamma$ is naturally isometric to the vertex set of the curve graph of the (only non-$\Sigma_{0,3}$ component of the) 
complement of $\gamma$. Similarly, in view of the discussion above about stabilisers, the link of $\bar\gamma$ in $X_{K}$ 
can be made into a metric space $L_{\bar\gamma}$ naturally isometric to the vertex set of the quotient of the curve graph of 
a complexity-1 surface by the action of the subgroup generated by $K$-th powers of 
Dehn twists supported on said surface.

 \begin{lemma}
 Let $\gamma$ be a curve. Let $\bar x,\bar y\neq \bar\gamma$ be adjacent vertices of $X_K$. Let $r(\bar 
x)=r_{\bar\gamma}(x),r(\bar y)=r_{\bar\gamma}(y)$ be the entrance points in $lk(\bar\gamma)$ of geodesics from $\bar x,\bar 
y$ to $\bar \gamma$. Then $d_{L_{\bar\gamma}}(r(\bar x),r(\bar y))\leq C$, where $C$ is as in Lemma 
\ref{lem:bgi_complexity_1}.
 \end{lemma}
 
 \begin{proof}
 For $\bar z$ a vertex of $X_K$, we choose any geodesic $\alpha_{\bar z}$ from $\bar z$ to $\bar\gamma$.
By Proposition \ref{prop:lift}, we can lift the geodesic triangle formed by $\alpha_{\bar x},\alpha_{\bar y}$ and an 
edge connecting $\bar x,\bar y$ to $\mathcal C(\Sigma_{g,p})$, and for convenience we arrange that one of the vertices of the 
lift is $\gamma$. 
 
 By Lemma \ref{lem:bgi_complexity_1}, we see that, in $\mathcal C(\Sigma_{g,p})$, the entrance points of the lifts of 
$\alpha_{\bar x},\alpha_{\bar y}$ in 
$lk(\gamma)$ are close to each other as measured in $L_\gamma$, which implies that the same holds for $r(\bar x),r(\bar y)$ 
in $L_{\bar\gamma}$.
 \end{proof}
 
 We now claim that, for any fixed $x_0\in lk(\bar\gamma)$, there exists $M$ with the following property. Suppose that we can 
write some $h\in H=Stab(\bar\gamma)$ as $h=s_1h_1\dots s_n$, where $s_i\in S,h_i\in H$ and $s_1h_1\dots s_i\notin H$ for all 
$i\leq n-1$. Then $d_{L_{\bar\gamma}}(x_0,hx_0)\leq Mn$. Since $L_{\bar\gamma}$ is equivariantly quasi-isometric to $H$ (see 
the argument for the hyperbolicity of $G(0,4,K)$ and $G(1,1,K)$ above), this concludes the proof of relative hyperbolicity.
 
 We can assume that $S$ is symmetric, and that $S\cap H=\emptyset$. We will use the maps $r_{\bar\gamma}$ (for any $\gamma$) 
from the lemma, 
 where we assume that the relevant geodesics are chosen equivariantly. 
 
Let $L_0=\max_{s\in S} \max_{\bar\gamma'}d_{L_{\bar\gamma'}}(r(x_0),r(s x_0))$. Notice that $L_0$ is indeed attained 
because for any given $s$ (of which there are finitely many), there are only finitely many $\bar\gamma'$ for which 
$d_{L_{\bar\gamma'}}(r(x_0),r(s x_0))$ can exceed $C$ times the distance between $x_0$ and $sx_0$, namely those 
occurring along a geodesic from $x_0$ to $sx_0$.

By equivariance, we have $d_{L_{\bar\gamma}}(r(gx_0),r(gs x_0))\leq L_0$  for any $g\in G(g,p,K)$. Moreover, if $g\notin H$ 
and $h'\in H$, then $gx_0$ and $ghx_0$ are connected by two edges not containing $\bar\gamma$, so that 
$d_{L_{\bar\gamma}}(r(gx_0),r(ghx_0))\leq 2C$. Hence
 $$d_{L_{\bar\gamma}}(x_0,hx_0)= d_{L_{\bar\gamma}}(r(x_0),r(hx_0))\leq L_0n+2(n-1)C,$$
 as required.
 
 Finally, in case~\eqref{item:medium}, hyperbolicity of $G(g,p,K)$ follows since a group hyperbolic relative to hyperbolic 
subgroups is hyperbolic by~\cite[Corollary 2.41]{Osin_relhyp}.
 \end{proof}

\subsection{Residual properties of $MCG(\Sigma_{g,p})$}\label{subsec:residual_properties}
Let $K_0,K$ be as in the statement of Theorem~\ref{thm:mcg_quotients} and let $G(g,p,K)$ and $X_K$ be as above.  Let 
$\Psi_K:MCG(\Sigma_{g,p})\to G(g,p,K)$ be the quotient map.

\begin{lemma}\label{lem:first_quotient_survive}
Let $x\in MCG(\Sigma_{g,p})\setminus\{1\}$.  Then $\Psi_K(x)\ne1$ for all 
sufficiently large $K\in K_0\mathbb Z$.
\end{lemma}

\begin{proof}
We consider three cases.
\begin{itemize}
 \item If $x$ is pseudo-Anosov, then it acts loxodromically on $\mathcal 
C(\Sigma_{g,p})$, so by Proposition \ref{prop:preserve_lox}, $\Psi_K(x)\ne1$ 
for 
sufficiently large $K$.
 \item If $x$ has finite order, then $x\not\in DT_K$ since each element of 
$DT_K$ has infinite order. This follows by transfinite induction using 
Theorem~\ref{theo;statement_that_should_have_been_in_D_PS}.(2),(4).

 \item The remaining possibility is that $x$ is reducible.  In this case, there 
exists $n>0$, depending only on $(g,p)$, so that $x^n$ stabilises a simple closed curve 
$\gamma$ of $\Sigma_{g,p}$.  Let $U,V$ be the components of 
$\Sigma_{g,p}\setminus\gamma$ (if $\gamma$ is non-separating, we take $V=\emptyset$ and $U=\Sigma_{g,p}\setminus\gamma$).  
Let $H\leq MCG(\Sigma_{g,p})$ be the stabiliser 
of $\gamma$.  Let $W=U\sqcup V$, which is a (possibly disconnected) subsurface 
of $\Sigma_{g,p}$ of complexity strictly less than that of $\Sigma_{g,p}$.

The action of $H$ on $W$ gives an exact sequence $1\to Z\to 
H\stackrel{\phi}{\longrightarrow} \widehat A$, 
where $Z$ is central in $H$ and is the cyclic subgroup generated by the Dehn 
twist about $\gamma$, and $\widehat A$ has a finite-index subgroup $A$ such 
that $A\leq MCG(U)$ if $V=\emptyset$ and $A\leq MCG(U)\times MCG(V)$ otherwise.

Let $H'\leq H$ be the finite-index subgroup 
$H'=\phi^{-1}(A)$.  Let $\Psi_K^U:MCG(U)\to G_K(U)$ be the quotient obtained by 
killing 
$K^{th}$ powers of Dehn twists in $MCG(U)$, and define $\Psi_K^V:MCG(V)\to 
G_K(V)$ analogously (if $V\neq\emptyset$).  Let $\Psi_K':A\to\overline A$ be the 
restriction of $\Psi_K^U$ to $A$ if $V=\emptyset$ and the restriction 
of $\Psi_K^U\times\Psi_K^V$ to $A$ otherwise.

Define a homomorphism $\bar \phi:\Psi_K(H')\to \overline A$ by 
$\bar\phi(\Psi_K(g))=\Psi_K'(\phi(g))$.  To see that this is well-defined, it 
suffices to show that $\Psi_K'(\phi(n))=1$ whenever $n\in DT_K\cap H'$.  By 
Proposition~\ref{prop:stab}, we have $n=yq$, where $y$ is a power of the Dehn 
twist about $\gamma$ and $q$ is the product of powers of Dehn twists about 
curves in $W$.  Moreover, $\phi(q)=\phi(yq)=\phi(n)\in A$, 
since $n\in H'$.  Hence $\phi(n)$ is the product of powers of Dehn twists in 
$U$ and $V$, and lies in $A$, so $\Psi'_K(\phi(n))$ is defined and 
$\Psi'_K(\phi(n))=1$, as required.

Choose $m\geq 1$ so that $x^{mn}\in H'$.  Write $x^{mn}=zw$, where $z\in Z$ and 
$w$ is supported on $W$, with either $x^{mn}=z$ or $\phi(w)\neq 1$.  

If $x^{mn}=z$, then $\Psi_K(x^{mn})=\Psi_K(z)\neq 1$ for sufficiently large 
$K$, by Proposition~\ref{prop:stab}.

Otherwise, $\phi(x^{mn})=\phi(w)\neq 1$.  Moreover, $\phi(w)\in A$, since 
$\phi(w)=\phi(x^{mn})$ and $x^{mn}\in H'$.  Hence either $\phi(w)=a\in MCG(U)$ 
(if $V=\emptyset$) or $\phi(w)=(a,b)\in MCG(U)\times MCG(V)$.  In either case, we 
can assume $a\neq 1$, so by induction on complexity, $\Psi_K^U(a)\neq 1$, and 
hence $\Psi_K'(\phi(w))\neq 1$, for all sufficiently large $K$, as required. 
 But 
$\Psi_K'(\phi(w))=\bar\phi(\Psi_K(x^{mn}))$, so $\Psi_K(x^{mn})\neq 1$, and 
hence $\Psi_K(x)\neq 1$.

(In the base case, $U$ is an annulus or pair of pants.  When $U$ is a pair of 
pants, $\Psi_K$ is the identity.  When $U$ is an annulus, $MCG(U)$ has a 
finite-index normal subgroup generated by a single Dehn twist, and the lemma 
clearly holds.)
\end{itemize}

By the Nielsen-Thurston classification, any $x$ is of one of the above three 
types, so the lemma holds.
\end{proof}

This, and Theorem~\ref{thm:mcg_quotients}, are already sufficient to prove:

\begin{cor}\label{cor:res_hyp}
 Suppose $(g,p)\in\{(0,4),(0,5),(1,0),(1,1),(1,2)\}$. Then $MCG(\Sigma_{g,p})$ is fully residually non-elementary 
hyperbolic.
\end{cor}

\begin{proof}
By Theorem~\ref{thm:mcg_quotients}, $G(g,p,K)$ is non-elementary 
hyperbolic, and by Lemma~\ref{lem:first_quotient_survive}, any finite subset of $MCG(\Sigma_{g,p})\setminus\{1\}$ is 
mapped injectively to $G(g,p,K)$ for all sufficiently large $K$.
\end{proof}

For convenience, whenever $g,p$ are fixed, we write $G_K$ to mean $G(g,p,K)$.  We now study images of convex-cocompact 
subgroups  in $G_K$.  

\begin{prop}\label{prop:convex_cocompact}
Let $(g,p)$ be as in Section~\ref{subsec:rel_hyp_quotient}, let $K_0,K$ be as in Theorem~\ref{thm:mcg_quotients}, and let 
$(G_K,\mathcal H)$ be the relatively hyperbolic structure from Theorem~\ref{thm:mcg_quotients}.  Let $Q\le 
MCG(\Sigma_{g,p})$ be a convex-cocompact subgroup.  Then:
\begin{enumerate}
     \item \label{item:injective} For all sufficiently large $K$, the quotient map $\Psi_K$ is injective on $Q$.
     \item\label{item:strongly_quasiconvex} For all sufficiently large $K$, $\Psi_K(Q)$ is strongly relatively quasiconvex 
in $(G_K,\mathcal H)$, and hence quasiconvex in $G_K$.
     \item\label{item:keep_out} For all $x\in MCG(\Sigma_{g,p})\setminus Q$ and all sufficiently large $K$, we have 
$\Psi_K(x)\not\in\Psi_K(Q)$.
\end{enumerate}
\end{prop}

\begin{proof}
Consider the action of $Q$ on $\mathcal C(\Sigma_{g,p})$ arising as the restriction of the action of $MCG(\Sigma_{g,p})$.   
Fix a vertex $v_0\in\mathcal C(\Sigma_{g,p})$.  Fix $h\in Q$, 
and let $s$ be a curve.  Then $d_s(v_0,hv_0)\leq C$ whenever the quantity is defined, for some uniform constant $C$: this is 
contained in the proof of~\cite[Theorem 7.4]{KentLeininger:convex_cocompact}, see \cite[Lemma 5.1]{DurhamTaylor:stability}. 
Choosing $K$ sufficiently 
large (in terms of $C$) and applying 
Lemma~\ref{lift;segments_witout_angles} implies that geodesics $[v_0,hv_0]\in\mathcal C(\Sigma_{g,p})$ map isometrically to 
geodesics in $X_K$ joining the images 
of $v_0,hv_0$.  This implies assertion~\eqref{item:injective}.  

Now, fix a finite generating set $\mathcal Y$ of $Q$.  Since geodesics $[v_0,hv_0]$ map isometrically to geodesics in $X_K$, 
we see that the orbit map $\Psi_K(h)\mapsto \Psi_K(h)v_0$ is a quasi-isometric embedding $\Psi_K(Q)\to X_K$ whose constants 
depend on $\mathcal Y$ but are independent of $K$.

By Theorem~\ref{thm:mcg_quotients}, $X_K$ is  $G_K$--equivariantly quasi-isometric to $Cay(G_K,\mathcal S\cup\mathcal 
H)$ for any finite relative generating set $\mathcal S$, so $\Psi_K(Q)\to Cay(G_K,\mathcal S\cup\mathcal 
H)$ is a quasi-isometric embedding.  Choosing $\mathcal S$ to be a 
finite generating set of $G_K$, we can pull back the quasi-isometric embedding $Q\to Cay(G_K,\mathcal S\cup\mathcal 
H)$ under the Lipschitz map $Cay(G_K,\mathcal S)\to Cay(G_K,\mathcal S\cup\mathcal 
H)$ to get a quasi-isometric embedding $Q\to Cay(G_K,\mathcal S)$, proving that $\Psi_K(Q)$ is quasiconvex in the hyperbolic 
group $G_K$.  This proves assertion~\eqref{item:strongly_quasiconvex}.  

Let $x\in MCG(\Sigma_{g,p})\setminus Q$.  Given $K$, let $\bar v_0$ be the image of $v_0$ in $X_K$.  Let 
$\Delta=d_{\mathcal C(\Sigma_{g,p})}(v_0,xv_0)$.  Then there exists $\Delta'$, depending only on 
$\Delta$ and the generating set of $Q$, such that for all sufficiently large $K$, the set of $h\in Q$ such that 
$d_{X_K}(\Psi_K(h)\bar v_0,\bar v_0)\leq\Delta$ is contained in the set $\{h_i\}$ of elements of $Q$ of word-length at most 
$\Delta'$.  This is because any orbit map $Q\to \Psi_K(Q)\to X_K$ is a quasi-isometric embedding with constants independent 
of $K$.

Suppose that $\Psi_K(x)\in \Psi_K(Q)$.  Since $d_{X_K}(\Psi_K(x)\bar v_0,\bar v_0)\le\Delta,$ we have 
$\Psi_K(xh_i^{-1})=1$ for some $i$.  For sufficiently large $K$, the 
map $\Psi_K$ is injective on $Q$, so $x=h_i^{-1}$, contradicting that $x\not\in Q$.  This proves 
assertion~\eqref{item:keep_out}.
\end{proof}

\begin{thm}
 Assume that all hyperbolic groups are residually finite.  Let 
$(g,p)\in\{(0,4),(0,5), (1,0),(1,1),(1,2) \}$.  Then any convex-cocompact subgroup $Q<MCG(\Sigma_{g,p})$ is 
separable in $MCG(\Sigma_{g,p})$.
\end{thm}

\begin{proof}
Let $G_K,\Psi_K$ and $X_K$ be as in Section~\ref{subsec:rel_hyp_quotient}.  Fix $x\in MCG(\Sigma_{g,p})\setminus Q$.  Using 
Theorem~\ref{thm:mcg_quotients} and Proposition~\ref{prop:convex_cocompact}, we can choose $K$ so that $G_K$ is 
hyperbolic, $\Psi_K(Q)$ is quasiconvex in $G_K$, and $\Psi_K(x)\not\in\Psi_K(Q)$.  If every hyperbolic group is residually 
finite, then by~\cite[Theorem 0.1]{AGM}, for any hyperbolic group, all of its quasiconvex subgroups 
are separable.  In particular, $G_K$ has a finite quotient separating $x$ from $\Psi_K(Q)$.  Hence there is a finite 
quotient 
of $MCG(\Sigma_{g,p})$ separating $x$ from $Q$, as required.
\end{proof}

\bibliographystyle{alpha}
\bibliography{quotients_MCG}

\begin{thebibliography}{BBFS17}

\bibitem[AAS07]{AAS:MCG_not_rel_hyp}
James~W. Anderson, Javier Aramayona, and Kenneth~J. Shackleton.
\newblock An obstruction to the strong relative hyperbolicity of a group.
\newblock {\em J. Group Theory}, 10(6):749--756, 2007.

\bibitem[AGM09]{AGM}
Ian Agol, Daniel Groves, and Jason~Fox Manning.
\newblock Residual finiteness, {QCERF} and fillings of hyperbolic groups.
\newblock {\em Geom. Topol.}, 13(2):1043--1073, 2009.

\bibitem[Ago13]{vhak}
I.~Agol.
\newblock The virtual {H}aken conjecture.
\newblock {\em Doc. Math.}, 18:1045--1087, 2013.
\newblock With an appendix by Agol, Daniel Groves, and Jason Manning.

\bibitem[AL17]{AlLa}
Aurelien Alvarez and Vincent Lafforgue.
\newblock Actions affines isométriques propres des groupes hyperboliques sur
  des espaces $\ell^p$.
\newblock {\em Expo. Math.}, 35(1):103--118, 2017.

\bibitem[BBF15]{BBF}
Mladen Bestvina, Ken Bromberg, and Koji Fujiwara.
\newblock Constructing group actions on quasi-trees and applications to mapping
  class groups.
\newblock {\em Publ. Math. Inst. Hautes \'{E}tudes Sci.}, 122:1--64, 2015.

\bibitem[BBFS17]{BBFS}
Mladen Bestvina, Ken Bromberg, Koji Fujiwara, and Alessandro Sisto.
\newblock Acylindrical actions on projection complexes.
\newblock {\em arXiv:1711.08722}, 2017.

\bibitem[BDM09]{BDM:thick}
Jason Behrstock, Cornelia Dru\c{t}u, and Lee Mosher.
\newblock Thick metric spaces, relative hyperbolicity, and quasi-isometric
  rigidity.
\newblock {\em Math. Ann.}, 344(3):543--595, 2009.

\bibitem[Beh06]{Beh:thesis}
Jason~A. Behrstock.
\newblock Asymptotic geometry of the mapping class group and {T}eichm\"{u}ller
  space.
\newblock {\em Geom. Topol.}, 10:1523--1578, 2006.

\bibitem[BHS15]{HHS_II}
Jason Behrstock, Mark~F Hagen, and Alessandro Sisto.
\newblock Hierarchically hyperbolic spaces {II}: combination theorems and the
  distance formula.
\newblock {\em arXiv preprint arXiv:1509.00632}, 2015.

\bibitem[BHS17a]{HHS:quasiflats}
J.~Behrstock, M.~F. Hagen, and A.~Sisto.
\newblock Quasiflats in hierarchically hyperbolic spaces.
\newblock {\em arXiv:1704.04271}, 2017.

\bibitem[BHS17b]{HHS_I}
Jason Behrstock, Mark~F. Hagen, and Alessandro Sisto.
\newblock Hierarchically hyperbolic spaces, {I}: {C}urve complexes for cubical
  groups.
\newblock {\em Geom. Topol.}, 21(3):1731--1804, 2017.

\bibitem[CC07]{CharneyCrisp}
Ruth Charney and John Crisp.
\newblock Relative hyperbolicity and {A}rtin groups.
\newblock {\em Geom. Dedicata}, 129:1--13, 2007.

\bibitem[Dah18]{D_PlateSpinning}
Fran{\c{c}}ois Dahmani.
\newblock The normal closure of big {D}ehn twists, and plate spinning with
  rotating families.
\newblock {\em Geom. Topol.}, 22:4113--4144, 2018.

\bibitem[Del96]{Del:kill_power}
Thomas Delzant.
\newblock Sous-groupes distingu\'{e}s et quotients des groupes hyperboliques.
\newblock {\em Duke Math. J.}, 83(3):661--682, 1996.

\bibitem[DG18]{DG:recognize_DF}
Fran\c{c}ois Dahmani and Vincent Guirardel.
\newblock Recognizing a relatively hyperbolic group by its {D}ehn fillings.
\newblock {\em Duke Math. J.}, 167(12):2189--2241, 2018.

\bibitem[DGO17]{DGO}
F.~Dahmani, V.~Guirardel, and D.~Osin.
\newblock Hyperbolically embedded subgroups and rotating families in groups
  acting on hyperbolic spaces.
\newblock {\em Mem. Amer. Math. Soc.}, 245(1156):v+152, 2017.

\bibitem[DT15]{DurhamTaylor:stability}
Matthew Durham and Samuel~J Taylor.
\newblock Convex cocompactness and stability in mapping class groups.
\newblock {\em Algebraic \& Geometric Topology}, 15(5):2837--2857, 2015.

\bibitem[DT19]{DT-DecidingIsom}
Fran\c{c}ois Dahmani and Nicholas Touikan.
\newblock Deciding isomorphy using {D}ehn fillings: the splitting case.
\newblock {\em Invent. Math.}, 2019.
\newblock to appear.

\bibitem[FM02]{FarbMosher}
Benson Farb and Lee Mosher.
\newblock Convex cocompact subgroups of mapping class groups.
\newblock {\em Geom. Topol.}, 6:91--152, 2002.

\bibitem[GM08]{GrMa-perfill}
D.~Groves and J.~F. Manning.
\newblock Dehn filling in relatively hyperbolic groups.
\newblock {\em Israel J. Math.}, 168:317--429, 2008.

\bibitem[Ham05]{Hamenstadt}
Ursula Hamenst{\"a}dt.
\newblock Word hyperbolic extensions of surface groups.
\newblock {\em arXiv preprint math/0505244}, 2005.

\bibitem[KL08]{KentLeininger:convex_cocompact}
Autumn~E Kent and Christopher~J Leininger.
\newblock Shadows of mapping class groups: capturing convex cocompactness.
\newblock {\em Geometric and Functional Analysis}, 18(4):1270--1325, 2008.

\bibitem[LM07]{LeiningerMcreynolds}
Christopher~J. Leininger and D.~B. McReynolds.
\newblock Separable subgroups of mapping class groups.
\newblock {\em Topology Appl.}, 154(1):1--10, 2007.

\bibitem[MM99]{MM_I}
Howard~A. Masur and Yair~N. Minsky.
\newblock Geometry of the complex of curves. {I}. {H}yperbolicity.
\newblock {\em Invent. Math.}, 138(1):103--149, 1999.

\bibitem[MM00]{MM_II}
H.~A. Masur and Y.~N. Minsky.
\newblock Geometry of the complex of curves. {II}. {H}ierarchical structure.
\newblock {\em Geom. Funct. Anal.}, 10(4):902--974, 2000.

\bibitem[Nic13]{Nica}
Bogdan Nica.
\newblock Proper isometric actions of hyperbolic groups on {$L^p$}-spaces.
\newblock {\em Compos. Math.}, 149(5):773--792, 2013.

\bibitem[Osi06]{Osin_relhyp}
Denis~V. Osin.
\newblock Relatively hyperbolic groups: intrinsic geometry, algebraic
  properties, and algorithmic problems.
\newblock {\em Mem. Amer. Math. Soc.}, 179(843):vi+100, 2006.

\bibitem[Osi07]{Os-perfill}
D.~V. Osin.
\newblock Peripheral fillings of relatively hyperbolic groups.
\newblock {\em Invent. Math.}, 167(2):295--326, 2007.

\bibitem[Osi16]{Osin:acyl}
D.~Osin.
\newblock Acylindrically hyperbolic groups.
\newblock {\em Trans. Amer. Math. Soc.}, 368(2):851--888, 2016.

\bibitem[Rei06]{Reid:separable}
Alan~W. Reid.
\newblock Surface subgroups of mapping class groups.
\newblock In {\em Problems on mapping class groups and related topics},
  volume~74 of {\em Proc. Sympos. Pure Math.}, pages 257--268. Amer. Math.
  Soc., Providence, RI, 2006.

\bibitem[Yu05]{Yu}
Guoliang Yu.
\newblock Hyperbolic groups admit proper affine isometric actions on
  {$l^p$}-spaces.
\newblock {\em Geom. Funct. Anal.}, 15(5):1144--1151, 2005.

\end{thebibliography}
\end{document}